\documentclass[12pt,twoside,reqno]{amsart}

\usepackage{amsfonts,amssymb,amsmath,mathrsfs}
\usepackage{lineno,hyperref}
\usepackage{comment}
\usepackage{mathtools}
\usepackage{color}
\usepackage{xcolor}
\usepackage{amssymb}
\usepackage{cite}
\usepackage{calc}
\usepackage{mathabx}

\usepackage[paperheight=10.50in,paperwidth=8.25in,margin=1in,heightrounded]{geometry}

\DeclareRobustCommand{\SkipTocEntry}[5]{}

\newcommand*{\ii}{\mathrm{i}}	
\newcommand{\x}{\langle x\rangle}
\newcommand{\csi}{\langle \xi \rangle}
\newcommand{\pdd}{\langle D \rangle}
\newcommand{\jap}{\langle \cdot\rangle}
\newcommand*{\caO}{\ensuremath{\mathcal{O}}}		
\newcommand*{\caL}{\ensuremath{\mathcal{L}}}

\newcommand{\gh}{\mathcal{G}^{\sigma}\mathcal{H}_{r,\rho}}

\newcommand{\R}{\mathbb{R}}
\newcommand{\N}{\mathbb{N}}
\newcommand{\Z}{\mathbb{Z}}
\newcommand{\T}{\mathbb{T}}
\newcommand{\cinf}{\mathscr{C}^{\infty}}
\newcommand{\Cc}{\mathscr{C}}
\newcommand{\Sc}{\mathscr{S}}
\newcommand{\Fc}{\mathscr{F}}
\newcommand{\C}{\mathbb{C}}
\newcommand{\Op}{\mathrm{Op}}
\newcommand{\ind}{\mathrm{ind}}
\def\Op{ {\operatorname{Op}} }
\newcommand*{\caS}{\ensuremath{\mathscr{S}}}		
\newcommand*{\scrL}{\ensuremath{\mathscr{L}}}	
\newcommand*{\cZ}{\ensuremath{\mathcal{Z}}}		

\newcommand{\nnorm}[1]{\|\hspace*{-1.3pt}|#1|\hspace*{-1.3pt}\|}
\newcommand{\bignnorm}[1]{\left\|\hspace*{-1.3pt}\left|#1\right|\hspace*{-1.3pt}\right\|}
\newcommand{\condA}{$\mathrm{(}\mathscr{A}\mathrm{)}$}
\newcommand{\condB}{$\mathrm{(}\mathscr{B}\mathrm{)}$}

\theoremstyle{definition}
\newtheorem{defn}{Definition}[section]
\theoremstyle{plain}
\newtheorem{thm}[defn]{Theorem}
\newtheorem{prop}[defn]{Proposition}
\newtheorem{lemma}[defn]{Lemma}
\newtheorem{cor}[defn]{Corollary}
\theoremstyle{remark}
\newtheorem{rem}[defn]{Remark}

\title[Global hypoellipticity and solvability in time-periodic weighted spaces]
       {Global hypoellipticity and solvability\\for a class of evolution operators\\in time-periodic weighted Sobolev spaces}

\author[F. Avila Silva]{Fernando de \'Avila Silva}
\address{Department of Mathematics  \\ Federal University of Paran\'a \\ Curitiba, CEP 81531-980, Caixa Postal 19081, Paran\'a, Brazil}
\email{fernando.avila@ufpr.br}

\author[M. Bonino]{Matteo Bonino}
\address{Dipartimento di Matematica ``G. Peano'' \\Universit\`a di Torino\\
	Via Carlo Alberto 10\\
	10123 Torino\\
	Italy}
\email{matteo.bonino@unito.it}

\author[S. Coriasco]{Sandro Coriasco}
\address{Dipartimento di Matematica ``G. Peano'' \\Universit\`a di Torino\\
	Via Carlo Alberto 10\\
	10123 Torino\\
	Italy}
\email{sandro.coriasco@unito.it}

\begin{document}
	
	\begin{abstract}
		We study the hypoellipticity and solvability properties of a class of time-periodic evolution operators, 
		with coefficients globally defined on $\R^d$ and growing polynomially with respect to 
		the space variable. To this aim, we introduce a class of time-periodic weighted Sobolev 
		spaces, whose elements are characterised in terms of suitable Fourier expansions,
		associated with elliptic operators. 
	\end{abstract}


	\maketitle

	\tableofcontents

	\section{Introduction}\label{sec:intro}
	\setcounter{equation}{0}

In this work we introduce a new class of \textit{Gevrey time-periodic weighted} Sobolev spaces and study 
hypoellipticity and solvability properties of a naturally associated class of evolution operators, with coefficients 
growing polynomially with respect the space variable.

In order to introduce our main results and techniques, let 
$\T^n \sim\R^n /(2\pi \Z)^n$ be the $n$-dimensional torus and fix the parameters  $\sigma > 1$ and $r,\rho \in \R$.  
Given a constant $C>0$, we first study smooth functions
$u: \T^n \to \Sc'=\Sc^\prime (\R^d)$ satisfying 
\begin{equation}\label{eq:tpsk}
	\sup _{\gamma \in \Z_+^n}\left\{ C^{-|\gamma|} (\gamma!)^{-\sigma} \sup _{t \in \T^n} \| \partial_t^\gamma u(t)\|_{H^{r,\rho}(\mathbb{R}^d)}\right\} < + \infty.
\end{equation}
In \eqref{eq:tpsk}, for arbitrary $r,\rho \in \R$, the Sobolev-Kato (or \textit{weighted Sobolev}) spaces are defined as
\begin{equation}\label{eq:skspace}
  	H^{r,\rho}=H^{r,\rho}(\R^d)= \{v \in \Sc^\prime(\R^{d}) \colon \|v\|_{r,\rho}=
	\|{\jap}^r\pdd^\rho v\|_{L^2}< \infty\},
\end{equation}
where $\pdd^\rho=\Op(\lambda_\rho)$ is the pseudo-differential operator with symbol $\lambda_{\rho}(\xi)=\csi^\rho$,
the power $\rho$ of the so-called \textit{japanese bracket} $\csi=\sqrt{1+|\xi|^2}$. In the sequel we will also employ the symbols
\[
	\lambda_{r,\rho}(x,\xi)=\lambda_r(x)\lambda_\rho(\xi)=\x^r\csi^\rho,\quad r,\rho\in\R.
\]
The space $H^{r,\rho}$ is equipped with the naturally induced norm $\|\cdot\|_{r,\rho}$. Endowed with such norm, $H^{r,\rho}$ is a Hilbert space. 
Moreover, $H^{r,\rho} \subseteq H^{\ell,\lambda}$, $\ell \leq r, \lambda \leq \rho$, with continuous embedding, compact when the order 
components inequalities are both strict. Recall that $\displaystyle\Sc'=\bigcup_{r,\rho\in\R}H^{r,\rho}$, respectively 
$\Sc=\displaystyle\bigcap_{r,\rho\in\R}H^{r,\rho}$, is the inductive,
respectively projective, limit of the Sobolev-Kato spaces.

A main aspect of our analysis consists in the characterisation of the spaces $H^{r,\rho}$ in terms of a
discretization approach, based on the Fourier expansions given by suitable (SG-)elliptic operators on $\R^d$ 
(see \cite{cordes} and Appendix \ref{sec:sgcalc} for details). Such operators are defined by means of symbols 
satisfying global estimates, namely
\begin{equation}
	\label{eq:disSG}
	|\partial_x^{\alpha} \partial_\xi^{\beta} a(x, \xi)| \leq C_{\alpha\beta} 
	\x^{m-|\alpha|}\csi^{\mu-|\beta|},
	\qquad (x, \xi) \in \R^d \times \R^d,
\end{equation}
for any $\alpha,\beta\in\Z^d$ and suitable $C_{\alpha\beta}>0$, that is, belonging to the classes $S^{m,\mu}=S^{m,\mu}(\R^d)$, 
$m,\mu\in\R$. Notice that, when $a\in S^{m,\mu}$, the corresponding operator $\Op(a)$ maps continuously 
$H^{r,\rho}$ to $H^{r-m,\rho-\mu}$, $r,\rho,m,\mu\in\R$ (see Appendix \ref{sec:sgcalc}). Ellipticity here means that there exists 
$R\ge0$ such that
\begin{equation}\label{eq:sgell}
	C\x^{m} \csi^{\mu}\le |a(x,\xi)|,\qquad 
	|x|+|\xi|\ge R, (x,\xi)\in\R^d\times\R^d,
\end{equation}
for some constant $C>0$. Our characterisation of the spaces $H^{r,\rho}$ is an extension to this setting of analogous 
results for the unweighted Sobolev spaces $H^\sigma$ and for the Sobolev spaces $Q^{\sigma}$ associated with 
the so-called Shubin calculus (see, e.g., \cite{GPR} and the references quoted therein).

Moreover, given an operator $P=\Op(p)$ on $\R^d$, associated with an elliptic symbol $p\in S^{m,\mu}(\R^d)$ and
satisfying suitable assumptions (see Section \ref{sec:eigenexp} below), we use its associated orthonormal basis of 
eigenfunctions $\{\phi_j\}_{j \in \N^\ast}$ to obtain 
\begin{equation}\label{eq:expansion-intro}
	u(t) = \sum_{j \in \N^\ast} u_j(t) \phi_j,
\end{equation}
for $u$ as in \eqref{eq:tpsk} and $u_j(t)$, $j \in \N^\ast, \N^\ast=\N \setminus \{0\}=\{1,2,3,\dots\}$, a suitable sequence of (Gevrey) periodic functions.

We apply our results on the spaces of functions satisfying \eqref{eq:tpsk} to the study of \textit{global hypoellipticity} and \textit{global solvability} for operators of the type 
\begin{equation}\label{eq:op-intro}
	L= D_t + \omega\Op(p) , \quad  D_t = -i\partial_t, t\in\T=\T^1, \omega \in \C, p\in S^{m,\mu}.
\end{equation}
Namely, expansions lead us to study the equation $Lu=f$ by means of a sequence of ordinary differential equations,
\begin{equation}\label{eq:odesintro}
	D_tu_j(t) + \lambda_j \omega u_j(t) = f_j(t), \ t \in \T, \ j \in \N^\ast,
\end{equation}
where $\{\lambda_j\}_{j \in \N^\ast}$ is the sequence of eigenvalues of $\Op(p)$. Hence, the regularity of 
the solution $u(t)$ is studied by analysing the behaviour of the sequence of the coefficients
$u_j(t)$ and their derivatives as $j\to \infty$, cf. Theorem \ref{thm:MainTheormHypo} below. In a similar fashion,
the solvability properties of $L$ are investigated as well, cf. Theorem \ref{thm:MainTheormSolv} below.

We emphasize that many authors explore characterisations of functional spaces in terms of Fourier expansions generated by elliptic operators. First, we mention Seeley's papers \cite{See65,See69}  in  the context of smooth and analytic functions on vector bundles. On the  Hilbert spaces and closed manifolds we may cite \cite{DR} by Delgado and Ruzhansky.  Similar ideas, in case of ultradifferentiable classes on compact manifolds and  Lie groups, are presented by Dasgupta and Ruzhanky in \cite{DaR}, by Kirilov, Moraes and  Ruzhanky in \cite{KMR2022,KMR2020}, see also the paper by Greenfield and Wallach \cite{GW73}. 
In the context of Gelfand-Shilov classes in the Euclidian spaces, a similar approach is introduced by Cappiello, Gramchev and Pilipovi\'c in \cite{GPR,CGPR}.

Furthermore, expansions like \eqref{eq:expansion-intro} on the product of manifolds (compact or not) have been recently explored and applied to studying certain periodic evolution equations. For instance,  \'Avila, Gramchev, and Kirilov consider in \cite{AGK18} an analysis for smooth functions and distributions on $\T \times M$, where $M$ is a smooth closed manifold, applied to investigations of the global hypoellipticity of such equations. In \cite{KMR2020,KMR2022}, the authors extend these ideas to the context of compact Lie groups and their products. The non-compact case $\T\times \R^d$ is investigated by \'Avila and Cappiello in \cite{AC,AC24} in the setting of  Gelfand-Shilov classes with applications to  hypoellipticity and solvability problems. Extensions of such theories and approaches to the setting $\T\times X$, where $X$ is a manifold with ends or, more generally, 
an asymptotically Euclidean manifold, are currently being investigated. They will be the subject of future papers, as well as the case of time-dependent operators of the form
$L=D_t+\omega(t)\Op(p)$, $p\in S^{m,\mu}$.

We recall that Hounie presented in \cite{Hou79} the study of global properties for the evolution  operator $\mathcal{L} = D_t + A$, where $A$ is a suitable operator, densely defined on a separable complex Hilbert space $\mathcal{H}$, which is unbounded, positive, and with eigenvalues diverging to $+\infty$. It was proved by Hounie that the structure of the spectrum of $A$ does not play a role in the regularity of solutions when $t$ is in some interval in $\R$, but it is instead crucial in the periodic case $t \in \T$.

This influence is connected with the so-called \textit{Diophantine approximations} and is closely related to Liouville numbers and generalizations. This is widely explored for studying global properties of operators on the torus, see, for instance, \cite{BER,DGY,GW72,Pet05,GPR} and the references quoted therein. Broadly speaking, this phenomenon appears due to the periodicity of equation \eqref{eq:odesintro} (see expressions \eqref{Solu-1-Constant} and \eqref{Solu-2-Constant} below). This connection is analysed in each one of the references cited above and is also important here, cf. Definition \ref{def:condA}.

The paper is organized as follows. In Section \ref{sec:eigenexp} we prove our first main results, namely,
the characterisation of the Sobolev-Kato spaces $H^{r,\rho}$ and of the Schwartz space $\Sc$ by means of Fourier expansions with
respect to the sequence $\{\phi_j\}_{j \in \N^\ast}$ 
of the eigenfunctions of an elliptic 
operator $\Op(p)$, $p\in S^{m,\mu}$, under suitable assumptions. In the subsequent Section \ref{sec:torussgexp}
we introduce and study Gevrey time-periodic Sobolev-Kato spaces on $\T^n\times\R^d$ and their duals, and prove our second set of main results,
characterising their elements again by means of Fourier expansions associated with the sequence $\{\phi_j\}_{j \in \N^\ast}$. Finally, in Sections \ref{sec:torussghypoell}	
and \ref{sec:torussgsolv}	
we prove our last main results, studying the hypoellipticity and solvability, respectively,
of operators of the form \eqref{eq:op-intro}, globally on $\T\times\R^d$. For the convenience of
the reader, we also include an Appendix, where we recall basic facts about 
the SG-calculus and the spectral properties of the elliptic SG-operators. In the sequel, the notation $A\asymp B$ means that $A\lesssim B$ and $B\lesssim A$,
where $A\lesssim B$ means that $A\le c\cdot B$, for a suitable constant $c>0$.

\addtocontents{toc}{\SkipTocEntry}
\section*{Acknowledgements}
Thanks are due to Proff. M. Cappiello, L. Neyt, L. Rodino, S. Pilipovi\'c, and J. Vindas, for useful discussions.
The first author gratefully acknowledges the support provided by the National Council for Scientific and Technological Development--CNPq, Brazil 
(grants 423458/2021-3, 402159/2022-5, 200295/2023-3 and 305630/2022-9). 
The first author also expresses gratitude for the hospitality extended to him during his visit to the 
Department of Mathematics ``G. Peano'', University of Turin, Italy, where part of this work was developed.
The second and third author have been partially supported by INdAM - GNAMPA Project 
CUP\_E53C22001930001 (Sc. Resp. S. Coriasco). The third author has been partially supported by 
the Italian Ministry of the University and Research - MUR, within the framework of the Call relating to the scrolling of the final rankings of the 
PRIN 2022 - Project Code 2022HCLAZ8, CUP D53C24003370006 (PI A. Palmieri, Local unit Sc. Resp. S. Coriasco).

\section{Eigenfunction expansions in weighted Sobolev spaces on $\R^d$}\label{sec:eigenexp}
	\setcounter{equation}{0}
	In this section we aim at characterising the Sobolev-Kato spaces by means of the eigenfunction expansions 
	related to a suitable elliptic, normal SG-operator $P$. This approach was considered by Seeley \cite{See69,See65} 
	in the context of smooth and analytic functions on vector bundles, and by Cappiello, Gramchev, Pilipovi\'c and Rodino in 
	\cite{GPR,CGPR} in Gelfand-Shilov classes on the Euclidean spaces. Here we precisely characterise smoothness
	and polynomial decay of temperate distributions on $\R^d$, in terms of the behaviour at infinity of the coefficients of 
	eigenfunction expansions, related to the corresponding behaviour of the eigenvalues of $P$. As a byproduct, we also
	similarly characterise when a temperate distribution is actually a rapidly decreasing Schwartz function.


	Let $P \in \mathrm{Op}(S^{m,\mu})$ be an elliptic, normal SG-operator with order components $m,\mu>0$
	(see Appendix \ref{sec:sgcalc} below for the notation and a summary of the theory of SG-operators).
	By \cite[Theorem 4, Ch. 3]{cordes}, we have that $\ker P = \ker P ^* \subset \Sc$. Moreover, the ellipticity of $P$ implies that
	it is Fredholm, and we have $\dim\ker P = N < \infty$ and $\ind\, P= \dim\ker P - \dim\ker P^*=0$.
	
	Let $\{\phi_j\}_{j=1}^N \subset \Sc$ be an orthonormal basis of $\ker P$ and 
	consider the operator $P_0$ with kernel
	$$
	K(x,y) = \sum_{j=1}^N \phi_j(x) \overline{\phi_j(y)} \in \Sc(\R^d\times\R^d),
	$$
	that is, the orthogonal projection on $\ker P$. We then get 
	$$
	P_0 u = \int K(x,y)u(y)dy = \sum_{j=1}^N  (u,\phi_j)\,\phi_j ,
	$$
	so that $P_0 \in \Op(S^{-\infty,-\infty})$ and it is a compact operator $P_0\colon L^2 \to L^2$.
	
	Notice that the operator $\widetilde{P} =P+P_0 \in \Op(S^{m,\mu})$ is elliptic, normal and injective with 
	$\ind\,\widetilde{P} = \ind\, P$. Therefore, $\widetilde P\colon H^{m,\mu} \to L^2$ is bijective with inverse 
	$Q \in \caL(L^2, H^{m,\mu})$. In particular, $Q\in\Op(S^{-m,-\mu})$ is normal and compact. Therefore, we have proved 
	that the next Proposition \ref{prop:eqnorws} holds true (see, e.g., \cite[Ch. 3]{NR}, (3.1.8), (3.1.9), for the last statement).

	\begin{prop}\label{prop:eqnorws}
		Let $P \in \mathrm{Op}(S^{m,\mu}(\R^d))$ be an elliptic, normal SG-operator with order components $m,\mu>0$. 
		Denote by $P_0$ the orthogonal projection on $\ker P$ and let $\widetilde P=P+P_0$. Then,
		$$
			u \in H^{m,\mu}(\R^d) \Longleftrightarrow P u \in L^2(\R^d)
		$$
		and
		$$
			\|u\|_{H^{m,\mu}(\R^d)} \doteq \|\Op(\lambda_{m,\mu})u \|_{L^2(\R^d)} \asymp \|\widetilde Pu\|_{L^2(\R^d)}.
		$$
	\end{prop}

	Since $Q\colon L^2 \to L^2$ is a compact normal operator, 
	there exists a basis $\{\phi_j\}_{j \in\N^\ast}$ of orthonormal eigenfunctions, associated with eigenvalues $\{\mu_j\}_{j \in\N^\ast}$
	such that
	$$
	\mu_j \to 0,  \quad j \to \infty.
	$$
    	The injectivity of $Q$ ensures that $\mu_j \neq 0$, for all $j \in \N^\ast$. We claim that $\phi_j$ is still an eigenfunction 
	of $\widetilde P$, with eigenvalue $\widetilde{\lambda}_{j} = \mu_j^{-1}$. Indeed, 
    $$
    Q \phi_j = \mu_j \phi_j,
    $$
    and, since $\phi_j \in \mathrm{Im}\, Q =  H^{m,\mu} = \mathrm{dom}\, \widetilde P$, and $\widetilde{P }Q = \mathrm{Id}$ on $H^{m,\mu}$, we get
    $$
    \widetilde P \phi_j = \mu_j^{-1} \phi_j = \lambda_{j}\phi_j.
    $$
    We can of course choose an orthornormal basis $\{\phi_j\}_{j \in \N^\ast}$ which completes the basis $\{\phi_j\}_{j=1}^{N}$ of $\ker P$, so that,
    clearly, 
    $$
    Q \phi_j = \mu_{j} \phi_j \iff \mu_{j}^{-1} \phi_j = \widetilde P \phi_j = \phi_j\Rightarrow \widetilde \lambda_j=\mu_j^{-1}=1, \quad j=1, \ldots, N,
    $$
    and
    \[
      Q \phi_j = \mu_{j} \phi_j \iff \mu_{j}^{-1} \phi_j = \widetilde P \phi_j = P\phi_j\Rightarrow \widetilde \lambda_j=\mu_j^{-1}=\lambda_j, \quad j\ge N+1.  
    \]
    Then, the eigenvalues $\widetilde{\lambda}_j$ of $\widetilde{P}$ are given by
 \begin{equation}\label{eq:wtlambda}
 	\widetilde \lambda_j=
 		\begin{cases}
 			1, \quad j=1, \dots, N, \\
 			\lambda_j \quad j\ge N+1,
 		\end{cases}
 \end{equation}
 where $\{\lambda_j\}_{j\ge N+1}$ are the nonvanishing eigenvalues of $P$ and, of course, $\lambda_1=\lambda_2=\cdots=\lambda_N=0$.
 Finally, $\displaystyle\lim_{j\to\infty}\mu_j=0\Rightarrow\lim_{j\to\infty}|\lambda_j|=\infty$.

    \begin{defn}
    	Let $P \in \mathrm{Op}(S^{m,\mu}(\R^d))$ be an elliptic, normal SG-operator with order components $m,\mu>0$, 
	and denote by $\{\phi_j\}_{j \in \N^\ast}$ a basis of orthonormal eigenfunctions of $P$. 
	Given $f \in \Sc(\R^d)$ we set
	\[
    		f_j = (f,\phi_j)_{L^2(\R^d)},\quad j \in \N^\ast,
    	\]
    	which implies $f = \sum_{j \in \N^\ast} f_j \phi_j$. By duality, for $u \in  \Sc^\prime(\R^d)$ we set 
    	\begin{equation}\label{eq:uj}
    		u_j = u\!\left(\overline{\phi_j}\right)=\langle u , \overline{\phi_j}\rangle,
    	\end{equation}
    	which implies $u = \sum_{j \in \N^\ast} u_j \phi_j$.
	\end{defn}
	\begin{rem}
		 Indeed, since $\{\overline{\phi_j}\}_{j \in \N^\ast}$ is also a basis for $L^2$, for any $\psi\in\Sc(\R^d)$
		\[
    			\langle u , \psi \rangle \!=\! \left\langle u, \sum_{j \in \N^\ast}(\psi,\overline{\phi_j})\,\overline{\phi_j} \right\rangle 
			\!=\! \sum_{j \in \N^\ast}   \langle u,\overline{\phi_j}\rangle \left(\int_{\R^d}\hspace*{-1.5mm}\psi \phi_j \right)
			\!=\! \left\langle \sum_{j \in \N^\ast} u_j\phi_j, \psi\right\rangle .
    		\]
	\end{rem}
    
	The next Theorem \ref{thm:exp_Hmmu} is our first main result. We show that a temperate distribution $u$ belongs 
	to a certain Sobolev-Kato space if and only if its associated Fourier coefficients $\{u_j\}_{j \in \N^\ast}$ satisfy
	a certain behaviour for $j\to\infty$, in relation with the eigenvalues $\{\lambda_j\}_{j \in \N^\ast}$ of $P$.
	
	\begin{thm}\label{thm:exp_Hmmu}
		Let $P \in \mathrm{Op}(S^{m,\mu}(\R^d))$ be an elliptic, normal SG-operator 
		with order components $m,\mu>0$, and denote by $\{\phi_j\}_{j \in \N^\ast}$ a basis of orthonormal 
		eigenfunctions of $P$ with corresponding eigenvalues $\{\lambda_j\}_{j \in \N^\ast}$. Let $r \in \N$. Then, $u \in \Sc^\prime(\R^d)$
		belongs to $H^{rm,r\mu}(\R^d)$ if and only if 
		\[
			\sum_{j \in \mathbb N} |u_j|^2 |\lambda_j|^{2r} <\infty,
		\]
		with $u_j$ defined in \eqref{eq:uj}. Moreover, 
		\begin{equation}\label{eq:Hmmuseries}
			\|u\|_{H^{rm,r\mu}(\R^d)}^2 \asymp \sum_{j \in \N^\ast} |u_j|^2 | \widetilde{\lambda}_j|^{2r},
		\end{equation}
		with the eigenvalues $\{\widetilde{\lambda}_j\}_{j \in \N^\ast}$ of $\widetilde{P} = P + P_0$, $P_0$ the projection on $\ker P$,
		given in \eqref{eq:wtlambda}.
	\end{thm}
\begin{proof}
Let $u = \sum_{j \in \N^\ast}  u_j \phi_j \in \Sc^\prime$. Then, 
\begin{align*}
\|P^r u\|^2_{L^2} & = \sum_{j \in \N^\ast}  \left| \left(P^r u , \phi_j\right)\right|^2
 =  \sum_{j \in \N^\ast}  \left|\sum_{k \in \N}  u_k \left(P^r \phi_k , \phi_j\right)\right|^2 \\
& =  \sum_{j \in \N^\ast}  \left|\sum_{k \in \N}  u_k \lambda_k^r \left( \phi_k , \phi_j\right)\right|^2 
=  \sum_{j \in \N^\ast} |u_j|^2 |\lambda_{j}|^{2r}.
\end{align*}
By normality of $P$, and the fact that $P_0$ is a projection,
\[
	PP_0=P_0P=0\Rightarrow\widetilde{P}^r=(P+P_0)^r=P^r+P_0^r=P^r+P_0,
\]
and the claims follow, in view of Proposition \ref{prop:eqnorws}. 
\end{proof}
\begin{rem}\label{rem:normphij}
	Notice that \eqref{eq:Hmmuseries} actually holds true for $u\in\Sc'$, $r\in\Z$. Indeed, for $r\in\Z$, $r<0$, 
	we have $\|u\|_{H^{rm,r\mu}}\asymp\|Q^r u\|_{L^2}$. Incidentally, we also observe that,
	\[
		N_{2}|\widetilde{\lambda}_j|^r \le \|\phi_j\|_{H^{rm,r\mu}}\le N_{1} |\widetilde{\lambda}_j|^r, \quad j \in \N^\ast,r\in\Z,
	\]
	with constants $N_1,N_2>0$ depending only on $P,r,d,m,\mu$. 
\end{rem}
\begin{cor}\label{cor:schw}
	Under the same hypotheses of Theorem \ref{thm:exp_Hmmu}, we have, for $u \in \Sc^\prime(\R^d)$,
	\[
		u \in \Sc(\R^d)\iff \sum_{j \in \N^\ast} |u_j|^2 |\lambda_j|^{2M} <\infty \text{ for any $M\in\N$}.
	\]
\end{cor}
\begin{proof}
	Immediate, by Theorem \ref{thm:exp_Hmmu}, in view of the equality $\Sc=\displaystyle\bigcap_{r,\rho\in\R}H^{r,\rho}$
	(cf. Appendix \ref{sec:sgcalc} and \cite[Ch. 3]{cordes}).
\end{proof}
%
%
The next Corollary \ref{cor:exp_Hmmu} follows by combining Theorem \ref{thm:exp_Hmmu} with 
results concerning the spectral asymptotics of self-adjoint SG-classical operators (cf. their summary 
in Appendix \ref{subs:sgasym}) and the properties of their fractional powers
(see, e.g., \cite{MSS06}, \cite[Ch. 4]{NR} and the references quoted therein).
\begin{cor}\label{cor:exp_Hmmu}
	Let $P \in \mathrm{Op}(S_\mathrm{cl}^{m,\mu}(\R^d))$ be an elliptic, invertible, self-adjoint, positive, classical SG-operator 
	with order components $m,\mu>0$, and denote by $\{\phi_j\}_{j \in \N^\ast}$ a basis of orthonormal 
	eigenfunctions of $P$ with corresponding eigenvalues $\{\lambda_j\}_{j \in \N^\ast}$. Let $r \in \R$. Then, $u \in \Sc^\prime(\R^d)$
	belongs to $H^{rm,r\mu}(\R^d)$ if and only if 
	\[
		\begin{cases}
			\displaystyle\sum_{j \in \mathbb N} |u_j|^2 \, j^\frac{2r\min\{m,\mu\}}{d} <\infty, & \text{ if $m\not=\mu$,}
			\\
			\displaystyle\sum_{j \in \mathbb N} |u_j|^2 \left(\frac{j}{\log j}\right)^\frac{2rm}{d} <\infty, & \text{ if $m=\mu$,}
		\end{cases}
	\]
	with $u_j$ defined in \eqref{eq:uj}. Moreover, 
	\begin{equation}\label{eq:Hmmuseriesbis}
		\|u\|_{H^{rm,r\mu}(\R^d)}^2 \asymp \sum_{j \in \N^\ast} |u_j|^2 | {\lambda}_j|^{2r} \asymp
				\begin{cases}
			\displaystyle\sum_{j \in \mathbb N} |u_j|^2 \, j^\frac{2r\min\{m,\mu\}}{d} <\infty,  &\text{if $m\not=\mu$,}
			\\
			\displaystyle\sum_{j \in \mathbb N} |u_j|^2 \left(\frac{j}{\log j}\right)^\frac{2rm}{d} <\infty,  &\text{if $m=\mu$.}
		\end{cases}
	\end{equation}
\end{cor}
\begin{proof}
	The operator $P^r$ is self-adjoint and it holds $P^r\in\Op(S^{rm,r\mu}_\mathrm{cl})$ (cf. \cite{MSS06}).
	It has eigenfunctions $\{\phi_j\}_{j \in \N^\ast}$ with eigenvalues $\{\lambda_j^r\}_{j \in \N^\ast}$, and satisfies
	$P^ru=\sum_{j \in \N^\ast}\lambda_j^ru_j\phi_j$. It is then bijective
	as an operator $P^r\colon H^{rm,r\mu}\to L^2$, so that, for any $u\in H^{rm,r\mu}$,
	$\|u\|_{H^{rm,r\mu}}\asymp \|P^ru\|_{L^2}$. The claims follow by 
	$\displaystyle\|P^ru\|^2_{L^2}=\sum_{j \in \N^\ast}|u_j|^2|\lambda_j|^{2r}$, as in Theorem \ref{thm:exp_Hmmu}, 
	taking into account \eqref{eq:sgeigenasm} 
	in Appendix \ref{subs:sgasym}.
\end{proof}
\begin{rem}\label{rem:normphijbis}
	Under the hypotheses of Corollary \ref{cor:exp_Hmmu}, we observe that,
	for $r\in\Z$, $j\to\infty$,
	\begin{equation}\label{eq:normphijasymp}
			\|\phi_j\|_{H^{rm,r\mu}} \asymp | {\lambda}_j|^{r} \asymp
			\begin{cases}
			 j^\frac{r\min\{m,\mu\}}{d},  & \text{ if $m\not=\mu$,}
			\\
			\displaystyle \left(\frac{j}{\log j}\right)^\frac{rm}{d},  & \text{ if $m=\mu$,}
		\end{cases}
	\end{equation}
	where the constants tacitly involved in \eqref{eq:normphijasymp} depend only on $P,r,d,m,\mu$.
\end{rem}
The next, simple Lemma \ref{lemma:h-s} is a property of the kernel of the powers $P^M$ of a
SG-operator of negative orders, for sufficiently large $M\in\N$. 
For the sake of completeness, we recall its proof in Appendix \ref{sec:sgcalc}. 
\begin{lemma}\label{lemma:h-s}
	Let $P \in \mathrm{Op}(S^{m,\mu}(\R^d))$ be an SG-operator with order components $m,\mu<0$. Then
	$P^M$ is Hilbert-Schmidt for any $M \in \N$ such that $\max\{Mm, M\mu\} <- \frac d2$. 
\end{lemma}

We now prove our second main result, namely, a characterisation of the elements of the space $\Sc$ by means
of uniform convergence of series generated by SG-operators. This is an analog of Theorem 10.2 in \cite{See65}.
\begin{thm}\label{thm:unifconver}
	Let $P \in \mathrm{Op}(S_\mathrm{cl}^{m,\mu}(\R^d))$ be an elliptic, normal operator 
	with order components $m,\mu>0$ or $m,\mu<0$, and denote by $\{\phi_j\}_{j \in \N^\ast}$ a basis of orthonormal 
	eigenfunctions of $P$. Then, $f\in\Sc(\R^d)$ if and only if 
	\begin{equation}\label{eq:convunsch}
		\sum_{j \in \N^\ast} \left|(f,\phi_j)_{L^2(\R^d)} \right| \, \left|(A\phi_j)(x) \right|, \quad x\in\R^d, 
	\end{equation}
	converges uniformly on $\R^d$ for every SG-operator $A$.
\end{thm}

\begin{proof}
	Notice that, since \eqref{eq:convunsch} only involves the eigenfunctions of $P$, it is possible to consider
	$\widetilde{P}=P+P_0\in \Op(S^{m,\mu})$ in place of $P$: the eigenfunctions are the same, with eigenvalues
	$\{\widetilde{\lambda}_j\}_{j \in \N^\ast}$. Then, it is no restriction to assume $P$ invertible, and even
	with order components $m,\mu<0$. Indeed, if that is not the case, since $\widetilde{P}$ has no 
	vanishing eigenvalues, we can instead use $Q=(P+P_0)^{-1}\in\Op(S^{-m,-\mu})$, which 
	has again the same eigenfunctions, with eigenvalues $\{\widetilde{\lambda}_j^{-1}\}_{j \in \N^\ast}$.
	Having negative order components, as an operator from $L^2$ to itself, $P$ is then
	compact (cf. Appendix \ref{sec:sgcalc}).
	Finally, $P^M$, $M\in\N$, has again the same eigenfunctions of $P$, $P+P_0$, and $(P+P_0)^{-1}$, with eigenvalues
	$\{{\lambda}_j^{M}\}_{j \in \N^\ast}$. Summing up, for this proof we can assume that $P$ is also
	Hilbert-Schmidt (cf. Lemma \ref{lemma:h-s}), that is, 
	\[
		m,\mu<-\frac{d}{2}\Rightarrow\displaystyle\sum_{j \in \N^\ast}|\lambda_j|^2<\infty, \quad\lambda_j\not=0, j \in \N^\ast,
	\]
	and the topology of $H^{Mm,M\mu}$ is given
	by the equivalent norm $\|P^Mu\|_{L^2}$, $M\in\Z$, $M<0$. 
	
	Recall that, by the embeddings of the Sobolev-Kato spaces (cf. \cite[Ch. 3]{cordes} and Appendix \ref{sec:sgcalc}) and Sobolev's Lemma,
	for $t\ge0$, $\tau > \frac d2$, one has $\Sc\hookrightarrow H^{t, \tau} \hookrightarrow H^{0, \tau} \equiv H^\tau\hookrightarrow L^\infty$. 
	
	Assume $A\in\Op(S^{pm,p\mu})$, $p\in\Z$, $p<0$. 
	Again, this is no restriction, given the inclusions among
	the SG symbols and operators spaces (cf. Appendix \ref{sec:sgcalc}). It follows that, for any $j\in\Z$, $u\in H^{(j+p)m,(j+p)\mu}$,
	\[
		\|Au\|_{H^{jm,j\mu}}\lesssim \|u\|_{H^{(j+p)m,(j+p)\mu}},
	\]
	and, by the observations above,
	\begin{align*}
			|(A\phi_j)(x) | \leq \|A\phi_j\|_{L^\infty} &\lesssim \|A\phi_j\|_{H^{-m,-\mu}}
			\lesssim \|\phi_j\|_{H^{m(p-1),\mu(p-1)}} 
			\\
			&\asymp \|P^{p-1}\phi_j\|_{L^2}
			= |{\lambda}_j|^{p-1}.
	\end{align*}
 	Now, in view of Corollary \ref{cor:schw}, the sequence $\{|(f,\phi_j)| \, |\lambda_j|^{p-3}\}_{j \in \N^\ast}$ is bounded for any $f\in\Sc$: indeed,
	its elements are the terms of a convergent series. We then conclude
	\begin{align*}
	 	\sum_{j \in \N^\ast} \left|(f,\phi_j) \right| \, \left|(A\phi_j)(x) \right| &\lesssim
	 	\sum_{j \in \N^\ast} \left|(f,\phi_j)\right|\,\left| \lambda_j\right|^{p-3} \, \left| \lambda_j \right|^{2} 
		\\
		&\lesssim\sum_{j \in \N^\ast} |\lambda_j|^2<\infty, \quad x\in\R^d.
	\end{align*}
	We have showed that, for any $f\in\Sc$ and any SG-operator $A$, 
	\eqref{eq:convunsch} is totally convergent, which implies the first part of the claim.
	
	\smallskip
	Conversely, assume that \eqref{eq:convunsch} converges uniformly for every operator $A\in\Op(S^{s,\sigma})$, $s,\sigma\in\R$. Then
	we can take $A=M^\alpha D^\beta\colon\Sc\to\Sc$, where $D^\beta=(-i)^{|\beta|}\partial^\beta$ and, for
	$\phi\in\Sc$, $(M^\alpha \phi)(x)=x^\alpha \phi(x)$, $x\in\R^d$, $\alpha,\beta\in\Z_+^d$, so that $A\in\Op(S^{|\alpha|,|\beta|}_\mathrm{cl})$.
	In principle, we assume $f\in\Sc^\prime$ and the coefficients $(f,\phi_j)$ understood in the sense of \eqref{eq:uj}.
	However, defining 
	\begin{equation}\label{eq:convunschbis}
		\hspace*{-2mm}f_{\alpha\beta}(x)=\sum_{j \in \N^\ast} (f,\phi_j) \, (M^\alpha D^\beta\phi_j)(x)
					= \sum_{j \in \N^\ast} (f,\phi_j) \, x^\alpha D^\beta\phi_j(x),  x\in\R^d, 
	\end{equation}
	the hypotheses clearly imply that $f_{\alpha\beta}\in\Cc$, and
	actually, setting $g=f_{00}$, it also follows $g\in\cinf$ and $M^\alpha D^\beta g = f_{\alpha\beta}$,
	$\alpha,\beta\in\Z_+^d$. Moreover, by the hypothesis of uniform convergence on $\R^d$ of \eqref{eq:convunsch}, we also see that 
	for any $\alpha,\beta\in\Z_+^d$ there exists $\nu_{\alpha\beta}\in\N$ such that, for any $N>\nu_{\alpha\beta}$,
	\begin{align*}
		\sup_{x\in\R^d}\left| x^\alpha D^\beta g(x) \phantom{ \sum_{j=1}^N}\right.\hspace*{-5mm}
		&- \left.\sum_{j=1}^N  (f,\phi_j) \, x^\alpha D^\beta\phi_j(x) \right|
		\\
		&= \sup_{x\in\R^d}\left| \sum_{j\ge N+1}  (f,\phi_j) \, x^\alpha D^\beta\phi_j(x) \right| \le 1.
	\end{align*}
	Since $\phi_j\in\Sc$, $j \in \N^\ast$, it follows that, for any $\alpha,\beta\in\Z_+^d$, choosing $N>\nu_{\alpha\beta}$,
	there exists $C_{\alpha\beta}>0$ such that
	\begin{align*}
		\sup_{x\in\R^d} |x^\alpha D^\beta g(x)| &\le \sup_{x\in\R^d}\left|\sum_{j=1}^N (f,\phi_j)\, x^\alpha D^\beta\phi_j(x)\right|
		\\
		&+\sup_{x\in\R^d}\left| \sum_{j\ge N+1}  (f,\phi_j)\, x^\alpha D^\beta\phi_j(x) \right|\le C_{\alpha\beta},
	\end{align*}
	that is, $g\in\Sc\subset L^2$. Since $g$ and $f$ have the same Fourier coefficients $(f,\phi_j)$, $j \in \N^\ast$,
	we conclude $f=g\in\Sc$, as claimed.
	
	\smallskip
	\noindent
	The proof is complete.
\end{proof}

	\section{Gevrey-periodic weighted Sobolev spaces on $\T^n\times\R^d$}\label{sec:torussgexp}
	\setcounter{equation}{0}

	Here we introduce our classes of Gevrey time-periodic Sobolev-Kato spaces and study their properties.
	The main results are characterisations by means of eigenfunction expansions generated by an elliptic, normal SG-operator,
	see below Theorem \ref{thm:reprF} and its corollaries, Theorem \ref{thm:unifconvert}, and Theorem \ref{thm:seq}.
	
	\subsection{Gevrey classes on the torus}\label{subs:gevreytorus}
	Let us begin recalling the standard characterisation of the Gevrey classes on the $n$-dimensional torus $\T^n$. Given $\eta>0$ and $\sigma \geq 1$, 
	$\mathcal{G}^{\sigma,\eta}=\mathcal{G}^{\sigma,\eta}(\T^n)$ denotes the space of all smooth functions $u \in \cinf(\T^n)$ 
	such that there exists $C=C_{\sigma\eta}>0$ for which 
	\[
		\sup_{t \in \T^n} |\partial^{\gamma}_t u(t)| \leq C \eta^{|\gamma|}(\gamma!)^{\sigma}, \quad \gamma \in \Z_+^n.
	\]
	The space $\mathcal{G}^{\sigma,\eta}$ is a Banach space endowed with the norm
	\[
		\|u\|_{\mathcal{G}^{\sigma,\eta}} = 
		\sup_{\gamma \in \Z_+^n}\left \{\sup_{t \in \T^n} \eta^{-|\gamma|}(\gamma!)^{-\sigma} |\partial_t^{\gamma}u(t)| \right\}.
	\]
	Then, the space of periodic Gevrey functions of order $\sigma$ is defined by
	\[
		\mathcal{G}^{\sigma}(\T^n) =\displaystyle \underset{\eta\rightarrow +\infty}{\mbox{ind} \lim} \;\mathcal{G}^{\sigma, \eta} (\T^n).
	\]
	The dual space $(\mathcal{G}^{\sigma}(\T^n))'$ is defined as the set of all linear maps 
	$\theta\colon \mathcal{G}^{\sigma}(\T^n) \to \C$ with the property that (cf. \cite{RodinoGevreyBook}) for every $\eta>0$ there exists $C=C_\eta>0$ such that,
	for any $u \in \mathcal{G}^{\sigma,\eta}$,
	\[
		|\langle \theta , u \rangle| \leq C
		\sup_{\gamma \in \Z_+^n}\left\{\sup_{t\in \T^n} \eta^{-|\gamma|} (\gamma!)^{-\sigma}
		|\partial^\gamma_t u(t)|\right\} = C \|u\|_{\mathcal{G}^{\sigma,\eta}}.
	\]
	\subsection{Inductive limits of periodic Gevrey-Sobolev-Kato-type Banach spaces and their duals}
	We now define families of Banach spaces of smooth maps on the torus, 
	taking values in Sobolev-Kato (or weighted Sobolev) spaces, and satisfying Gevrey-type estimates
	with respect to the variable $t\in\T^n$.  
	\begin{defn}\label{def:gscrr}
		Let $\sigma > 1$ and $r,\rho \in \R$ be fixed. Given a constant $C>0$ we denote by 
		\begin{equation}\label{GHsigma,C,r,rho}
		\gh^{C} \doteq \gh^{C}(\T^n \times \R^d)
		\end{equation}
		the space of all functions $u \in \mathscr{C}^{\infty}(\T^n; \mathscr{S}' (\mathbb{R}^d))$ such that 
		\begin{equation}\label{norm_H_r,rho,C}
			\|u\|_{\gh^{C}(\T^n \times \R^d)}=\|u\|_{\sigma,C,r,\rho}\doteq
			\sup _{\gamma \in \Z_+^n}
			\left\{ C^{-|\gamma|} (\gamma!)^{-\sigma} \sup _{t \in \T^n} \| \partial_t^\gamma u(t)\|_{H^{r,\rho}(\mathbb{R}^d)}\right\} < + \infty.
		\end{equation}
	\end{defn}
	\noindent	
If $u \in \gh^{C}$, then:
\begin{itemize}
	\item for any $\gamma \in \Z_+^n$ it is well defined the map 
	\begin{equation}\label{mapTtoH}
			\T^n \ni t \mapsto \partial_t^{\gamma} u(t) \in H^{r,\rho};
	\end{equation}
	in particular, $u(t) \in H^{r,\rho}$ for all $t\in \T^n$;
	
	\item for any $\gamma \in \Z_+^n$ we have 
	$$
 		\sup _{t \in \T^n} \| \partial_t^\gamma u(t)\|_{H^{r,\rho}} \leq C^{|\gamma|+1} (\gamma!)^{\sigma},
	$$
	implying a Gevrey estimate to \eqref{mapTtoH}.
\end{itemize}

We now investigate topological properties for these spaces. First, notice that $\gh^C$ is a Banach space endowed with the norm 
$\left\| \cdot \right\|_{\sigma, C, r,\rho}$  given by \eqref{norm_H_r,rho,C}. Moreover, we have the immediate inclusions
	\begin{equation}\label{sig<tho}
		\gh^C \subset \mathcal{G}^{\mu}\mathcal{H}^C_{r,\rho}, \ 1<\sigma \leq \mu,
	\end{equation}
	\begin{equation}\label{CleqD}
		\gh^C\subset \mathcal{G}^{\sigma}\mathcal{H}^{\widetilde{C}}_{r,\rho},  \  0<C\leq \widetilde{C},
	\end{equation}
	\begin{equation}\label{sig,C,sobolev}
		\mathcal{G}^{\sigma}\mathcal{H}^C_{r,\rho} \subset \mathcal{G}^{\sigma}\mathcal{H}^C_{t,\tau}, \  t \leq r \textrm{ and } \tau \leq \rho.
	\end{equation}
	For $\sigma>1$ and $r,\rho \in \R$ we define the spaces
	$$
	\mathcal{G}^{\sigma}\mathcal{H}_{r,\rho}=\mathcal{G}^{\sigma}\mathcal{H}_{r,\rho}(\T^n \times \R^d) \doteq \bigcup_{C>0} 
	\mathcal{G}^{\sigma}\mathcal{H}_{r,\rho}^C,
	$$
	and
	\begin{equation}\label{eq:defGHrrho}
    		\mathcal{G}\mathcal{H}_{r,\rho}=\mathcal{G}\mathcal{H}_{r,\rho}(\T^n \times \R^d) 
		\doteq  \bigcup_{\sigma>1} \mathcal{G}^{\sigma}\mathcal{H}_{r,\rho} = \bigcup_{\sigma>1} \left(\bigcup_{C>0} \gh^C \right),
	\end{equation}
	equipped  with the induced inductive limit topologies. In this way,
	$\mathcal{G}\mathcal{H}_{r,\rho}$ is an inductive limit of an inductive limit, which is (mildly) inconvenient. 
	However, in view of the next Theorem \ref{thm:LBspaces}, we may see it as an inductive limit of Banach spaces, that is, as an (LB)-space.
	\begin{thm}\label{thm:LBspaces}
		For each $r, \rho \in \R$ the equality 
		\begin{equation}\label{eq:GHrrho}
			\mathcal{G}\mathcal{H}_{r,\rho}(\T^n \times \R^d) = \bigcup_{\sigma>1} \mathcal{G}^{\sigma}\mathcal{H}_{r,\rho}^{\sigma-1}(\T^n \times \R^d),
		\end{equation}
		with the induced inductive limit topology in the right-hand side, holds true both among sets as well as topologically.
	\end{thm}
	
	\begin{proof}
		Since $ \mathcal{G}^{\sigma}\mathcal{H}_{r,\rho}^{\sigma-1} = \mathcal{G}^{\sigma}\mathcal{H}_{r,\rho}^{C_{\sigma}}$, 
		with $C_{\sigma} = \sigma -1 >0$, we see that the right-hand side of \eqref{eq:GHrrho}
		is a subset of the left-hand side of \eqref{eq:defGHrrho}.  Conversely, let $u \in \mathcal{G}\mathcal{H}_{r,\rho}$.
		Then, $u \in \mathcal{G}^{\sigma}\mathcal{H}_{r,\rho}^{C}$, for some $\sigma >1$ and $C>0$. If $C < \sigma -1$, then
		$$
		u \in \mathcal{G}^{\sigma}\mathcal{H}_{r,\rho}^{C} \subset 
		\mathcal{G}^{\sigma}\mathcal{H}_{r,\rho}^{\sigma -1},
		$$
		in view \eqref{CleqD}. On the other hand, if $\sigma \leq C+1$, we obtain
		$$
		u \in \mathcal{G}^{\sigma}\mathcal{H}_{r,\rho}^{C} \subset 
		\mathcal{G}^{C+1}\mathcal{H}_{r,\rho}^{C},
		$$
		in view of \eqref{sig<tho}. Hence, $u \in \mathcal{G}^{\delta}\mathcal{H}_{r,\rho}^{\delta-1}$, where 
		$\delta = C+1>1$. Therefore, we have the claimed equality as a set. 
		
		Now, let   $\{u_j\}$  be a convergent sequence in $\bigcup_{\sigma>1} \mathcal{G}^{\sigma}\mathcal{H}_{r,\rho}^{\sigma-1}$. Then, 
		$u_j \to 0$ in $\mathcal{G}^{\sigma}\mathcal{H}_{r,\rho}^{C_{\sigma}}$, with $C_{\sigma} = \sigma -1$, 
		implying the convergence in $\mathcal{G}\mathcal{H}_{r,\rho}$. 
		On the other hand, suppose now that $\{u_j\}$ is a convergent sequence 
		in $\mathcal{G}\mathcal{H}_{r,\rho}$ defined as in \eqref{eq:defGHrrho}, namely, there are $\sigma>1$ and $C>0$ such that
		$$
		\lim\limits_{j \to \infty}
		\left[
		\sup _{\gamma \in \Z_+^n} C^{-|\gamma|} (\gamma!)^{-\sigma} \sup _{t \in \T^n} \| \partial_t^\gamma u_j(t)\|_{H^{r,\rho}}\right] = 0.
		$$
		If $C \leq \sigma -1$, then
		$$
	   	u_j \in  \mathcal{G}^{\sigma}\mathcal{H}_{r,\rho}^{C}
		\subset \mathcal{G}^{\sigma}\mathcal{H}_{r,\rho}^{\sigma-1}, \ j \in \N^\ast,
		$$
		and
		\begin{align*}
			(\sigma-1)^{-|\gamma|} (\gamma!)^{-\sigma} \sup _{t \in \T^n} \| \partial_t^\gamma & u_j(t)\|_{H^{r,\rho}} \\ &
			\leq 
			C^{-|\gamma|} (\gamma!)^{-\sigma} \sup _{t \in \T^n} \| \partial_t^\gamma u_j(t)\|_{H^{r,\rho}},
		\end{align*}
		which implies
		$$
		\lim\limits_{j \to \infty}
		\left[
		\sup _{\gamma \in \Z_+^n} (\sigma-1)^{-|\gamma|} (\gamma!)^{-\sigma} 
		\sup _{t \in \T^n} \| \partial_t^\gamma u_j(t)\|_{H^{r,\rho}(\mathbb{R}^d)}\right] = 0.
		$$
		If $\sigma < C+1$, then 
		$$
		u_j \in  \mathcal{G}^{\sigma}\mathcal{H}_{r,\rho}^{C}
		\subset \mathcal{G}^{C+1}\mathcal{H}_{r,\rho}^{C}, \ j \in\N^\ast, 
		$$
		and
		\begin{align*}
			C^{-|\gamma|} (\gamma!)^{-(C+1)} \sup _{t \in \T^n} \| \partial_t^\gamma & u_j(t)\|_{H^{r,\rho}} \\ &
			\leq 
			C^{-|\gamma|} (\gamma!)^{-\sigma} \sup _{t \in \T^n} \| \partial_t^\gamma u_j(t)\|_{H^{r,\rho}},
		\end{align*}
		which implies 
		$$
		\lim\limits_{j \to \infty}
		\left[
		\sup _{\gamma \in \Z_+^n} C^{-|\gamma|} (\gamma!)^{-(C+1)} \sup _{t \in \T^n} \| \partial_t^\gamma u_j(t)\|_{H^{r,\rho}(\mathbb{R}^d)}\right] = 0.
		$$
		Therefore, $u_j \to 0$ in $\bigcup_{\sigma>1} \mathcal{G}^{\sigma}\mathcal{H}_{r,\rho}^{\sigma-1}$, as desired. The proof is complete. 
	\end{proof}
	In view of Theorem \ref{thm:LBspaces}, from now on we consider $\mathcal{G}\mathcal{H}_{r,\rho}$ equipped with 
	the (equivalent) inductive limit topology defined by \eqref{eq:GHrrho}. 
	In particular, we denote by 
	$\mathcal{G}\mathcal{H}_{r,\rho}'=\mathcal{G}\mathcal{H}_{r,\rho}'(\T^n \times \R^d)=(\mathcal{G}\mathcal{H}_{r,\rho}(\T^n \times \R^d))'$ 
	the space of all linear and continuous maps 
	$\theta\colon \mathcal{G}\mathcal{H}_{r,\rho} \to \C$. Then, for every $\sigma>1$, there exists $A=A_{ \sigma}>0$ such that 
	$$
	|\langle \theta  , u \rangle| \leq A
	\sup_{\gamma \in \Z^n_+}  \left\{(\sigma-1)^{-|\gamma|}(\gamma!)^{-\sigma}\sup _{t \in \T^n}\| \partial_t^\gamma u(t)\|_{H^{r,\rho}}\right\},
	$$
	for every 
	$u \in \mathcal{G}^{\sigma}\mathcal{H}_{r,\rho}^{\sigma-1}$, 
	which is equivalent to the assertion that, for every $\sigma >1, C>0$, there exists $B=B_{\sigma C}>0$ such that 
	$$
	|\langle \theta  , u \rangle| \leq B
	\sup_{\gamma \in \Z^n_+}  \left\{C^{-|\gamma|}(\gamma!)^{-\sigma}\sup _{t \in \T^n}\| \partial_t^\gamma u(t)\|_{H^{r,\rho}}\right\},
	$$
	for every $u \in \mathcal{G}^{\sigma}\mathcal{H}_{r,\rho}^{C}$. Therefore, the next Proposition \ref{prop:charGHrrho} holds true.
	\begin{prop}\label{prop:charGHrrho}
		A linear functional $\theta\colon \mathcal{G}\mathcal{H}_{r,\rho}(\T^n \times \R^d) \to \C$ is an element of $\mathcal{G}\mathcal{H}_{r,\rho}'(\T^n \times \R^d)$ 
		if and only if, for every $\sigma >1, C>0$, there exists $B=B_{\sigma C}>0$ such that, for any 
		$u \in \mathcal{G}^{\sigma}\mathcal{H}_{r,\rho}^{C}(\T^n \times \R^d)$,
		$$
		|\langle \theta  ,  u \rangle| \leq B
		\sup_{\gamma \in \Z^n_+}\left\{  C^{-|\gamma|}(\gamma!)^{-\sigma}\sup _{t \in \T^n}\| \partial_t^\gamma u(t)\|_{H^{r,\rho}(\mathbb{R}^d)}\right\}\!.
		$$
		\end{prop}
	We can actually describe $\mathcal{G}\mathcal{H}_{r,\rho}'$ as a projective limit, dual to the inductive limit definition \eqref{eq:GHrrho} 
	of $\mathcal{G}\mathcal{H}_{r,\rho}$.
	\begin{thm}\label{Thm_dualGHr,rho}
		For each $r, \rho \in \R$ we have
		$$
		\mathcal{G}\mathcal{H}_{r,\rho}'(\T^n \times \R^d) =  \bigcap_{\sigma>1} 
		(\mathcal{G}^{\sigma}\mathcal{H}_{r,\rho}^{\sigma-1}(\T^n \times \R^d))',
		$$		
		equipped with the projective limit topology.
	\end{thm} 
	
	\begin{proof}
		If $\theta \in \mathcal{G}\mathcal{H}_{r,\rho}'$, then 
		$\theta$ is well defined as linear map on $\mathcal{G}^{\sigma}\mathcal{H}_{r,\rho}$, for each $\sigma>1$. 
		Hence, for $\sigma>1$ and $C>0$ there is  $B=B_{\sigma C}>0$ such that 
		$$
		|\langle \theta  ,  u \rangle| \leq B
		\sup_{\gamma \in \Z^n_+} \left\{ C^{-|\gamma|}(\gamma!)^{-\sigma}\sup _{t \in \T^n}\| \partial_t^\gamma u(t)\|_{H^{r,\rho}}\right\},
		$$
		for all $u \in \mathcal{G}^{\sigma}\mathcal{H}_{r,\rho}^{C}$.  This shows, in particular, that the restriction of $\theta$ to 
		$\mathcal{G}^{\sigma}\mathcal{H}_{r,\rho}^C$ belongs to 
		$(\mathcal{G}^{\sigma}\mathcal{H}_{r,\rho}^{\sigma-1})'$.
		
		Conversely, let $\theta \in \bigcap_{\sigma>1} 
		(\mathcal{G}^{\sigma}\mathcal{H}_{r,\rho}^{\sigma-1})'$. Then, for each $\sigma>1$ we have 
		$\theta \in (\mathcal{G}^{\sigma}\mathcal{H}_{r,\rho}^{\sigma-1})'$ and 
		$\theta$ is well defined on $\mathcal{G}\mathcal{H}_{r,\rho}$. Moreover, if 
		$u, v\in\mathcal{G}\mathcal{H}_{r,\rho}$, there is
		$\delta$ such that $u,v \in \mathcal{G}^{\delta}\mathcal{G}\mathcal{H}_{r,\rho}$ and for every
		$\alpha \in \C$ we have
		$$
		\langle \theta \, , \, \alpha u + v\rangle = 
		\alpha \langle \theta \, , \, u \rangle +
		\langle \theta \, , \,  v\rangle.
		$$
		Then, $\theta$ is linear on $\mathcal{G}\mathcal{H}_{r,\rho}$. Also,
		given $\sigma>1$ and $C>0$, there is $B=B_{\sigma C}>0$ satisfying
		$$
		|\langle \theta \, , \, u \rangle| \leq B
		\sup_{\gamma \in \Z^n_+}\left\{  C^{-|\gamma|}(\gamma!)^{-\sigma}\sup _{t \in \T^n}\| \partial_t^\gamma u(t)\|_{H^{r,\rho}}\right\},
		\quad u \in \mathcal{G}^{\sigma}\mathcal{H}_{r,\rho}^{C}. 
		$$
		Therefore,  $\theta \in \mathcal{G}\mathcal{H}_{r,\rho}'$.
	\end{proof}

	\subsection{The spaces $\Fc=\Fc(\mathbb T ^n\times \R^d)$ and $\Fc'=(\Fc(\mathbb T ^n\times \R^d))'$} 
	We observe that, in view of \eqref{sig,C,sobolev}, we have
	$$
	\mathcal{G}\mathcal{H}_{r,\rho} \subset \mathcal{G}\mathcal{H}_{t,\tau}, 
	$$ 
	whenever $t \leq r$ and $\tau \leq \rho$.  Then, we may consider the space
	\begin{equation}\label{eq:defF}
		\Fc = \Fc(\T ^n\times \R^d) \doteq \bigcap_{r,\rho \in \R} \mathcal{G}\mathcal{H}_{r,\rho}(\T ^n\times \R^d),
	\end{equation}
	endowed with the projective limit topology.
	In particular, $f_j \to 0$, $j\to\infty$, in $\Fc$ if, and only if, for each $r, \rho \in \R$ there is $\sigma= \sigma_{r\rho} >1$ such that
	\[
		\lim\limits_{j \to \infty} f_j = 0  \ \textrm{ in } \ 
		\mathcal{G}^{\sigma}\mathcal{H}_{r,\rho}.
	\]
	By  $\Fc'$ we denote the space of all linear continuous functionals $\theta\colon \Fc \to \C$. It follows, by \cite[$(6)$, page 290]{Kothe}, that
	\begin{equation}\label{F'=UH'}
		\Fc'=\Fc'(\T ^n\times \R^d) = \bigcup_{r,\rho \in \R} \mathcal{G}\mathcal{H}_{r,\rho}'(\T ^n\times \R^d),
	\end{equation}	
	endowed with the inductive limit topology.
	
	\subsection{Eigenfunctions expansions in $\Fc$}
	In this section we characterise the elements of the space $\Fc$ in terms of the behaviour of their
	($t$-dependent) coefficients of the eigenfunctions expansion generated by a suitable SG-operator.
	As in the previous Section \ref{sec:eigenexp}, here we will often denote by
	$P\in \Op(S^{m,\mu})$ a normal, elliptic SG-operator with order components $m, \mu>0$,
	and by $\left\{\phi_j\right\}_{j \in \mathbb{N}} \subset \mathscr{S}$
	the associated orthonormal basis of eigenfunctions, with eigenvalues $\{\lambda_j\}_{j \in \N^\ast}$. 
	
	\begin{rem}\label{rem:polyboundlambda}
		Notice that if we also assume $P\in\Op(SG_\mathrm{cl}^{m,\mu})$, self-adjoint and positive,
		both in the case $m\not=\mu$ and in the case $m=\mu$, we have that,
		for a suitable choice of $\varrho, \varrho' >0$ and constants $K,K'>0$
		$$
		K' j^{\varrho'} \leq |\lambda_{j}| = \lambda_j \leq K j^{\varrho}, \ j \to \infty.
		$$
		Moreover, $\varrho, \varrho'$ can be chosen such that $\varrho -  \varrho' \leq \epsilon$, for any $\epsilon >0$ (cf. Appendix \ref{subs:sgasym}).
	\end{rem}
	If $f \in \Fc$, then for all $r,\rho \in \R$ there are $\sigma=\sigma_{r\rho}>1$ and $C=C_{r\rho}>0$ such that
	$$
	\|f\|_{r,\rho,\sigma,C}\doteq\sup _{\gamma \in \Z_+^n}\left\{ C^{-|\gamma|} (\gamma!)^{-\sigma} \sup _{t \in \T^n} \| \partial_t^\gamma f(t)\|_{H^{r,\rho}}\right\} 
	< + \infty.
	$$
	Since, for any $\gamma\in\Z_+^n$,
	$$
	\T^n \ni t \mapsto \partial_t^{\gamma}f(t) \in \bigcap_{r,\rho \in \R} H^{r,\rho} = \Sc,
	$$	
	we see that $f \in  \mathscr{C}^{\infty}(\T^n; \Sc)$ and for all  $r,\rho \in \R$ there are $\sigma>1$ and $C>0$ as above 
	such that, for all $\gamma \in \Z_+^n$,
	$$
	\sup _{t \in \T^n} \| \partial_t^\gamma f(t)\|_{H^{r,\rho}} \leq C^{|\gamma|+1} (\gamma!)^{\sigma}.
	$$	
	Given $f \in \Fc$, we then have, for any $t\in\T^n$, the 
	Schwartz function $f(t)\colon \R^d \to \C$. Its Fourier coefficients $f_j: \T^n \to \C$, given by
	\begin{equation}\label{partial_in_F}
		f_j(t)  \doteq (f(t), \phi_j), \quad j \in \N^\ast,
	\end{equation}
	are well-defined, we actually have $f_j \in \cinf(\T^n)$ and, for any $\gamma \in \Z_+^n$,
	\begin{equation}\label{eq:fgammaj}
		f^\gamma_j(t) \doteq (\partial_t^\gamma  f(t), \phi_j) = \partial_t^\gamma f_j(t).
	\end{equation}
	Indeed, denoting by $\nnorm{\cdot}_N$, $N=0,1,\dots$, the $N$-th seminorm of $\Sc$, namely,
	\[
		\nnorm{g}_N=\sum_{|\alpha+\beta|\le N}\sup_{x\in\R^d}|x^\alpha\partial^\beta g(x)|, \quad g\in\Sc,
	\]
	for any $f\in\Fc$, $\gamma\in\Z_+^n$, $N=0,1,\dots$, we of course have
	\[
		\nnorm{\partial_t^\gamma f(t)}_N\in\mathscr{C}(\T^n)\Rightarrow \sup_{t\in\T^n}\nnorm{\partial_t^\gamma f(t)}_N
		=\max_{t\in\T^n}\nnorm{\partial_t^\gamma f(t)}_N=S^\gamma_N<\infty.
	\]
	Then, 
	\[
		|\partial^\gamma_t f(t,x)| \le S^\gamma_{2N}(1+|x|^2)^{-N},
		\quad f\in\Fc, \gamma\in\Z_+^n, N\in\N, t\in\T^n, x\in\R^d,
	\]
	and \eqref{eq:fgammaj} follows by dominated convergence. We have shown, for any $f\in\Fc$, $\gamma\in\Z_+^n$, $t\in\T^n$,
	\begin{equation}\label{eq:tder}
		\partial_t^\gamma f(t) = \partial_t^\gamma\sum_{j \in \N^\ast}f_j(t)\phi_j = \sum_{j \in \N^\ast}f^\gamma_j(t)\phi_j =  \sum_{j \in \N^\ast}\partial_t^\gamma f_j(t)\phi_j,
	\end{equation}

	The next Theorem \ref{thm:reprF}, together with its Corollaries \ref{cor:estpow}, \ref{cor:seq}, \ref{cor:suffhypo}, is our third main result.
	\begin{thm}\label{thm:reprF}
		Let $P \in \mathrm{Op}(S^{m,\mu}(\R^d))$ 
		be a normal, elliptic SG-operator with order components $m,\mu>0$. Then $f \in \Fc(\T ^n\times \R^d)$ if and only if it can be represented as    
		\[
			f(t)=\sum_{j \in \N^\ast} f_j (t) \phi_j
		\]
		with $f_j(t)$, $j \in \N^\ast$, defined by \eqref{partial_in_F} satisfying the condition
		\begin{equation}\label{eq:s}
			\tag{*}
			\hspace*{-3mm}
			\begin{aligned}
				\text{for }& \text{every $M \in \N$ there exist $\sigma\!=\!\sigma_M\!>\!1$ and $C\!=\!C_M\!>\!0$ such that}
				\\
				&\sup_{t \in \T^n}  \sum_{j \in \N^\ast} \widetilde{\lambda}_j^{2M} |\partial_t^\gamma f_j(t)|^2
				=
				\sup_{t \in \T^n}  \sum_{j \in \N^\ast} \widetilde{\lambda}_j^{2M} |f^\gamma_j(t)|^2 \leq C^{2(|\gamma|+1)} (\gamma!)^{2 \sigma} ,
				\\
				\text{for }&\text{every $\gamma \in \Z_+^n$}.
			\end{aligned}
		\end{equation}
	\end{thm}
	\begin{proof}
		Let $f \in \Fc$. By \eqref{eq:tder}, for any $\gamma\in\Z_+^n$ we can write 
		\[
			\partial_t^\gamma f(t)=\sum_{j \in \N^\ast} f_j^\gamma(t) \phi_j,
		\]
		with $f^\gamma_j(t)$ in \eqref{eq:fgammaj}.
		Moreover, denoting by $P_0\in\Op(S^{-\infty,-\infty})$ the projection on $\ker P$ by the definition of $\Fc$ and Corollary \ref{cor:schw},
		\begin{align*}
			f \in \mathscr F \iff & \forall M \in \N \; f \in \mathcal{G}\mathcal{H}_{Mm,M\mu} \\
			\iff & \forall M \in \N\, \exists \sigma=\sigma_M>1, C=C_M>0 \text{ such that } \forall \gamma \in \Z_+^n\\
			& \sup_{t \in \T^n} \| \partial^\gamma_t f(t) \|_{H^{Mm,M\mu}}^2 \leq C^{2(|\gamma|+1)} (\gamma!)^{2 \sigma} \\
			\iff & \forall M \in \N\, \exists \sigma>1, C>0 \text{ such that } \forall \gamma \in \Z_+^n\\
			& \sup_{t \in \T^n} \| (P+P_0)^M \, \partial^\gamma_t f(t) \|_{L^2}^2 \leq C^{2(|\gamma|+1)} (\gamma!)^{2 \sigma} \\
			\iff & \forall M \in \N\,\exists \sigma>1, C>0 \text{ such that } \forall \gamma \in \Z_+^n\\
			& \sup_{t \in \T^n} \left[ \| P^M \, \partial ^\gamma_t f(t) + P_0^M \, \partial ^\gamma_t f(t) \|_{L^2}^2 \right] 
			\leq C^{2(|\gamma|+1)} (\gamma!)^{2 \sigma}\\
			\iff & \forall M \in \N\, \exists \sigma>1, C>0 \text{ such that } \forall \gamma \in \Z_+^n\\
			& \sup_{t \in \T^n} \left[ \left\|\sum_{j \ge N+1}   f_j^\gamma(t) P^M\phi_j \right\|_{L^2}^2 + 
			\left\|\sum_{j=1}^N  f_j^\gamma(t) P_0^M\phi_j \right\|_{L^2}^2 \right]\\
			&\leq C^{2(|\gamma|+1)} (\gamma!)^{2 \sigma},
		\end{align*}
		which yields
		\begin{align*}
			f \in \Fc\iff & \forall M \in \N\, \exists \sigma>1, C>0 \text{ such that } \forall \gamma \in \Z_+^n \\
			& \sup_{t \in \T^n} \left[ \sum_{j \ge N+1} |\lambda_j|^{2M} |f_j^\gamma(t)|^2  + 
			\sum_{j=1}^N |f_j^\gamma(t)|^2 \right] \leq C^{2(|\gamma|+1)} (\gamma!)^{2 \sigma} \\
			\iff & \forall M \in \N\, \exists \sigma>1, C>0 \text{ such that } \forall \gamma \in \Z_+^n \\
			& \sup_{t \in \T^n} \sum_{j \in\N^\ast} |\widetilde{\lambda}_j|^{2M} |f_j^\gamma(t)|^2  \leq C^{2(|\gamma|+1)} (\gamma!)^{2 \sigma}, 
		\end{align*}
		as claimed.
	\end{proof}
	
	\begin{cor}\label{cor:estpow}
		Let $P\in\Op(S^{m,\mu}_\mathrm{cl}(\R^d))$ be a self-adjoint, positive, elliptic SG-operator with order components $m,\mu>0$.     
 		In this case, condition \eqref{eq:s} in Theorem \ref{thm:reprF} is equivalent to the condition
		\begin{equation}\label{eq:sclsa}
			\tag{**}
			\hspace*{-3mm}
			\begin{aligned}
				\text{for every  } &\text{$M \in \N$ there exist $\sigma\!=\!\sigma_M\!>\!1$ and $C\!=\!C_M\!>\!0$ such that}
				\\
				&\sup_{t \in \T^n} |\partial^\gamma_t f_j(t)| =
				\sup_{t \in \T^n} | f_j^\gamma(t)| \leq C^{|\gamma|+1} (\gamma!)^{\sigma} |\widetilde{\lambda}_j|^{-M},
				\\
				\text{for every } & j \in \N^\ast, \gamma \in \Z_+^n.
			\end{aligned}
		\end{equation}
	\end{cor}
	\begin{proof}
		The implication \eqref{eq:s} $\Rightarrow$ \eqref{eq:sclsa} is immediate, and it holds true
		even under only the less restrictive hypotheses of Theorem \ref{thm:reprF}. To show that the opposite
		implication \eqref{eq:sclsa} $\Rightarrow$ \eqref{eq:s} holds true as well,
		recalling Remark \ref{rem:polyboundlambda}, 
		we first observe that for any $M\in\N$ there exists $M'=M_M'\in\N$ such that
		\begin{equation}\label{eq:suml}
			0<\sum_{j \in \N^\ast} \widetilde{\lambda_j}^{2(M-M')} =  S = S_{MM'} < \infty.
		\end{equation}
		Then, for any $M \in \N$ we choose $M'\in\N$ satisfying \eqref{eq:suml} and, by hypothesis, there
		are $\sigma=\sigma_{M'}>1$ and $C=C_{M'}>0$ satisfying condition \eqref{eq:sclsa}, which implies
		\begin{align*}
			\sup_{t\in\T^n}\sum_{j \in \N^\ast} \widetilde{\lambda}_j^{2M} |f^\gamma_j(t)|^2
			& \leq \sum_{j \in \N^\ast} \widetilde{\lambda}_j^{2M}\sup_{t\in\T^n} |f^\gamma_j(t)|^2
			\le\sum_{j \in \N^\ast} \widetilde{\lambda}_j^{2M}C^{2(|\gamma|+1)} (\gamma!)^{2\sigma} \widetilde{\lambda}_j^{-2M'}  \\
			& =C^{2(|\gamma|+1)} (\gamma!)^{2\sigma} \sum_{j \in \N^\ast} \lambda_j^{2(M-M')}
			 = S C^{2(|\gamma|+1)} (\gamma!)^{2\sigma}.
		\end{align*}
		It follows that \eqref{eq:s} holds true for any $\gamma\in\Z_+^n$ with $\widetilde{C}$ in place of $C$,
		setting $\widetilde{C}=C$ if $S\le1$ or $\widetilde{C}=SC$ if $S>1$.
	\end{proof}
	\begin{cor}\label{cor:seq}
		Assume that $f \in \Fc(\T ^n\times \R^d)$. 
		\begin{itemize} 
			\item[ i)] Under the hypotheses of Theorem \ref{thm:reprF},
				for every $M \in \N$ there exist $\sigma=\sigma_M>1$ and $C=C_M>0$ such that
				\begin{equation}\label{eq-seq}
					\|f_j\|_{\mathcal{G}^{\sigma,C}(\T^n)} \leq C |\widetilde{\lambda}_j|^{-M}, \quad j \in \N^\ast.
				\end{equation}
			\item[ii)] Under the hypotheses of Corollary \ref{cor:estpow},  
				for every $M \in \N$ there exist $\sigma=\sigma_M>1$ and $C=C_M>0$ such that
				\[
					\|f_j\|_{\mathcal{G}^{\sigma,C}(\T^n)} \leq C j^{-M\varrho}, \ j \to \infty,
				\]
				for some $\varrho>0$. In particular, $\{f_j\}_{j \in \N^\ast} \subset \mathcal{G}^{\sigma_1}(\T^n)$. 
		\end{itemize}
\end{cor}
\begin{proof}
	i) By the proof of Theorem \ref{thm:reprF} and the definition of $\mathcal{G}^{\sigma,C}$ (cf. Section \ref{subs:gevreytorus}),
	it immediately follows that for
	every $M \in \N$ there exist $\sigma=\sigma_M>1$ and $C=C_M>0$ such that
	\[
		\sup_{t \in \T^n}|f_j^\gamma(t)| \leq C^{|\gamma|+1} (\gamma !)^{\sigma} |\widetilde{\lambda}_j|^{-M}, \quad j \in \N^\ast, \gamma \in \Z_+^n,
	\]
	that is,  $\{f_j\}_{j \in \N^\ast} \subset \mathcal{G}^{\sigma, C}$ and
	\[
		\|f_j\|_{\mathcal{G}^{\sigma,C}} \leq C |\widetilde{\lambda}_j|^{-M}, \quad j \in \N^\ast.
	\]
	ii) Recalling Remark \ref{rem:polyboundlambda}, the estimates for the Gevrey norms of the $f_j$ with respect to $j\to\infty$
	follow from the previous point. In particular, for  $M =1$, from $\lambda_{j} \to \infty$, we obtain $j_0\in \N$ such that	
	\[
		\sup_{t \in \T^n}|f_j^\gamma(t)| \leq  C_1 C^{|\gamma|+1} (\gamma !)^{\sigma_1}, \quad  j \geq j_0, \gamma \in \Z_+^n.
	\]
	On the other hand,
	\[
		\sup_{t \in \T^n}|f_j^\gamma(t)| \leq C_1 C'C^{|\gamma|+1} (\gamma !)^{\sigma_1}, \quad   j  =1, \ldots, j_0-1, \gamma \in \Z_+^n,
	\]
	where $C' = \max\{\widetilde{\lambda}_j^{-1}, \ j =1, \ldots, j_0-1\}$. 
	Therefore, $\{f_j\}_{j \in \N^\ast}\subset \mathcal{G}^{\sigma_1}$.
\end{proof}
	\begin{cor}\label{cor:suffhypo}
		Under the hypotheses of Corollary \ref{cor:estpow}, let $\{f_j\}_{j \in \N^\ast}\subset\mathcal{G}^{\sigma,C}(\T^n)$ 
		be a sequence with the property that for every $M' \in \N$ there exists $B=B_{M'}>0$ such that
		\[
			\|f_j\|_{\mathcal{G}^{\sigma,C}(\T^n)} \leq B j^{-M'}, \quad j \to \infty.
		\]
		Then, setting $f(t)=\sum_{j \in \N^\ast} f_j (t) \phi_j$, $t\in\T^n$, it follows $f \in \Fc(\T ^n\times \R^d)$.
	\end{cor}
	
	\begin{proof}
		For any $M \in \N$, choose $M'=M_M'\in \N$ so that $\sum_{j \in \N^\ast} j^{2(M\varrho -M')} < \infty$,
		recalling that, in view of Remark \ref{rem:polyboundlambda}, it holds $\widetilde{\lambda}_j\leq K j^{\varrho}$, $j\to\infty$. 
		Then, for a suitable $j_0\in\N$,
		\begin{align*}
			\sup_{t \in \T^n}  \sum_{j \ge j_0} \widetilde{\lambda}_j^{2M} |\partial^\gamma_t f_j(t)|^2
			& \leq \sum_{j \ge j_0} \lambda_j^{2M} C^{2(|\gamma|+1)} (\gamma!)^{2 \sigma}B^2j^{-2M'}  \\
			& \leq K^{2M}C^{2(|\gamma|+1)} (\gamma!)^{2 \sigma} B^2\sum_{j \ge j_0} j^{2M\varrho-2M'}  \\
			& \leq \widetilde{C}^{2(|\gamma|+1)}  (\gamma!)^{2 \sigma}, \quad \gamma\in\Z_+^n,
		\end{align*}
		which implies the claim.
	\end{proof}
	
	The next Theorem \ref{thm:unifconvert} is the analogue for the space $\Fc$ of Theorem \ref{thm:unifconver} for the Schwartz space $\Sc$:
	it is a consequence of Theorem \ref{thm:reprF} and its corollaries.
	\begin{thm}\label{thm:unifconvert}
		Let $P\in\Op(S^{m,\mu}(\R^d))$ be a normal, elliptic SG-operator with order components $m, \mu>0$ or $m,\mu<0$,
		and let $\{\phi_j\}$ be a corresponding orthonormal basis of eigenfunctions. 
		\begin{enumerate}
		\item\label{point:implltr} For any $f \in \Fc$ the series 
		\begin{equation}\label{eq:convunscht}
			\sum_{j \in \mathbb N} \left|\partial^\gamma_t(f(t),\phi_j)_{L^2(\R^d)} \right| \, |A\phi_j(x)| , \quad \gamma \in \Z_+^n,
		\end{equation}
		converge uniformly on $\T^n\times\R^d$ for every SG-operator $A$ and there exist $B=B_{APnd}>0$, 
		$\sigma=\sigma_{APnd}>1$, $C=C_{APnd}>0$, depending only on $A,P,n$, and $d$,
		such that the sums $S_{\gamma A}$ of \eqref{eq:convunscht} satisfy the condition
		\begin{equation}\label{eq:sumconvunift}
			S_{\gamma A}(t,x)\le B C^{|\gamma|+1} (\gamma!)^{\sigma}, \quad t\in\T^n, x\in\R^d.
		\end{equation}
		\item Additionally, assume $P\in\Op(S^{m,\mu}(\R^d))$, $m,\mu>0$, to be
		self-adjoint, positive and SG-classical. Then, for $f\in\cinf(\T^n,L^2(\R^d))$ the converse of the implication
		in point \eqref{point:implltr} above holds true as well.
		\end{enumerate} 
	\end{thm}
	\begin{proof}
		Arguing as in the proof of Theorem \ref{thm:unifconver}, we can assume $P$ invertible, with order components $m,\mu<-\frac{d}{2}$,
		and so Hilbert-Schmidt, with nonvanishing eigenvalues $\{	\lambda_j\}_{j \in \N^\ast}$ such that $\sum_{j \in \N^\ast}|\lambda_j|^2<\infty$, and
		$A\in\Op(S^{pm,p\mu})$, $p \in \Z$, $p<0$ so that $|A\phi_j(x) | \leq \widetilde{B} |\lambda_j|^{p-1}$, $x\in\R^d$, $\widetilde{B}>0$, and 
		$\widetilde{B}$ depends only on $A,P,n$, and $d$.
		\begin{enumerate}
		\item Assume $f\in\Fc$. Then, by Theorem \ref{thm:reprF}, for any $M\in\Z$, for suitable $\sigma>1$, $C>0$,
		\[
			\sup_{t \in \T^n} |\partial^\gamma_t (f(t),\phi_j)| \leq C^{(|\gamma|+1)} (\gamma!)^{ \sigma}|\lambda_j|^{-M}, \quad\gamma\in\Z_+^n.
		\]
		Choosing $M=p-3$, we then find, for any $t\in\T^n$, $x\in\R^d$, $\gamma\in\Z_+^n$,
		\begin{align*}
			\sum_{j \in \mathbb N} \left|\partial^\gamma_t(f(t),\phi_j) \right| \, |A\phi_j(x)|
			&\le
			\sum_{j \in \N^\ast} C^{(|\gamma|+1)} (\gamma!)^{ \sigma}|\lambda_j|^{-p+3} \widetilde{B} |\lambda_j|^{p-1}
			\\
			&\le  \widetilde{B}C^{(|\gamma|+1)} (\gamma!)^{ \sigma}\sum_{j \in \N^\ast}|\lambda_j|^2
			= BC^{(|\gamma|+1)} (\gamma!)^{ \sigma},
		\end{align*}
		which proves the claim.
		\item We start as in the second part of the proof of Theorem \ref{thm:unifconver}, setting $A=M^\alpha D^\beta$, $\alpha,\beta\in\Z_+^d$,
		\[
			f^\gamma_{\alpha\beta}(t,x)=\sum_{j \in \N^\ast}\partial^\gamma_t(f(t),\phi_j)\,(M^\alpha D^\beta\phi_j)(x),\quad\gamma\in\Z_+^n,
		\]
		and $g=f^0_{00}$. By the hypotheses, it follows $g\in\cinf(\T^n\times\R^d)$, with 
		$f^\gamma_{\alpha\beta}(t,x)=\partial_t^\gamma [x^\alpha D_x^\beta g(t,x)]$.  As in the proof of Theorem \ref{thm:unifconver}, 
		we actually find $g\in\cinf(\T^n,\Sc)$, and, of course, $g=f$. It then also follows 
		\[
			\partial^\gamma_t f_j(t)=\partial_t^\gamma(f(t),\phi_j)=(\partial_t^\gamma f(t),\phi_j)=f^\gamma_j(t),
			\; t\in\T^n,j \in \N^\ast,\gamma\in\Z_+^n.
		\]
		For any $M\in\N$ there exist $\sigma>1$, $C>0$, such that
		\begin{align*}
			&\|P^M\partial^\gamma_t f(t)\|_{L^2}\le C_d\bignnorm{\sum_{j \in \N^\ast} f^\gamma_j(t) P^M\phi_j}_{2\left[\frac{d}{2}\right]+2}
			\\
			&\le\widetilde{C}_d \sum_{|\alpha+\beta|\le2\left[\frac{d}{2}\right]+2}
			\sup_{\R^d}\left|\sum_{j \in \N^\ast}f^\gamma_j(t) \, M^\alpha D^\beta P^M\phi_j \right|
			\\
			&\le \widetilde{C}_d  \sum_{|\alpha+\beta|\le2\left[\frac{d}{2}\right]+2}
			\sup_{\R^d}\sum_{j \in \N^\ast}|f^\gamma_j(t)|\;|M^\alpha D^\beta P^M\phi_j|
			\\
			&\le \widetilde{C}_d  \sum_{|\alpha+\beta|\le2\left[\frac{d}{2}\right]+2}
			B_{\alpha\beta Pnd} \, C_{\alpha\beta Pnd}^{|\gamma|+1}\,(\gamma!)^{\sigma_{\alpha\beta Pnd}} 
			\le C^{|\gamma|+1}\,(\gamma!)^{\sigma}, t\in\T^n, \gamma\in\Z_+^n.
		\end{align*}
		It follows that for any $M\in\N$ there exist $\sigma>1$, $C>0$ such that
		\begin{align*}
			|\partial_t^\gamma f_j(t)|\, |\widetilde{\lambda}_j|^M &=
			|\partial_t^\gamma f_j(t)|\, |\widetilde{\lambda}_j|^M\,\|\phi_j\|_{L^2} = \|P^M\partial_t^\gamma  [(f(t),\phi_j)\phi_j]\|_{L^2}
			\\
			&\le \|P^M\partial_t^\gamma f(t)\|_{L^2} \le C^{|\gamma|+1}\,(\gamma!)^{\sigma}, \,\, t\in\T^n, \quad j \in \N^\ast, \gamma\in\Z_+^n
			\\
			&\iff \sup_{t\in\T^n}|\partial_t^\gamma f_j(t)|\le C^{|\gamma|+1}\,(\gamma!)^{\sigma} |\widetilde{\lambda}_j|^{-M}, \quad j \in \N^\ast,\gamma\in\Z_+^n.
		\end{align*}
		By Corollary \ref{cor:estpow}, this proves $f\in\Fc$.
		\end{enumerate}
	\end{proof}
	\begin{rem}\label{rem:conjugates}
		Since, for any $g\in\Sc'$, $\Op(\lambda_\rho)\overline{g}=\langle D\rangle^\rho \overline{g}=\overline{\langle D\rangle^\rho g}$,
		it also follows 
		\[	
			\jap^r \langle D\rangle^\rho \overline{g}=\overline{\jap^r \langle D\rangle^\rho g}
			\Rightarrow
			\|\overline{g}\|_{H^{r,\rho}}=\|g\|_{H^{r,\rho}}, \quad g\in\Sc', r,\rho\in\R.
		\] 
		Then, by Definition \ref{def:gscrr}, $u\in\gh^{C}\iff \overline{u}\in\gh^{C}$, with $\|u\|_{\sigma,C,r,\rho}=\|\overline{u}\|_{\sigma,C,r,\rho}$,
		for arbitrary $\sigma>1$, $C>0$, $r,\rho\in\R$, and, of course, $f\in\Fc\iff\overline{f}\in\Fc$. Recalling that also $\{\overline{\phi_j}\}_{j \in \N^\ast}$
		is an orthonormal basis, and observing that
		\[
			\widetilde{f_j}(t)\doteq(f(t),\overline{\phi_j})=\int_{\R^d} f(t)\,\phi_j=\overline{(\overline{f}(t),\phi_j)}=\overline{\overline{f}_j(t)}, \quad t\in\T^n,j \in \N^\ast,f\in\Fc,
		\]
		we notice that the statements of Theorem \ref{thm:reprF} and Corollaries \ref{cor:estpow}, \ref{cor:seq}, \ref{cor:suffhypo} hold true 
		also with the coefficients $\widetilde{f_j}$ in place of the coefficients $f_j$, $j \in \N^\ast$. This will
		be useful in the subsequent Section \ref{subs:reprfp}.
	\end{rem}
	
	\subsection{Eigenfunctions expansions in $\Fc'$}\label{subs:reprfp}
	Let	$\theta \in \Fc'$ and  $M \in \Z$ such that $\theta \in \mathcal{G}\mathcal{H}_{Mm,M\mu}'$. 
	For any $\psi \in \mathcal{G}^{\sigma}(\T^n)$ we consider $\psi \otimes \overline{\phi_j} \in \mathcal{G}^{\sigma}\mathcal{H}_{Mm,M\mu}$ by setting
	\begin{align*}
		\T^n \ni t \mapsto \psi(t)\overline{\phi_j} \colon  \R^d \to \C \colon x \mapsto \psi(t)\overline{\phi_j(x)}, \quad j \in \N^\ast.
	\end{align*}
	Then, there are well-defined linear maps $\theta_j: \mathcal{G}^{\sigma}(\T^n) \to \C, $ given by
	\begin{equation}\label{partial_in_F'}
		\langle  \theta_j  ,  \psi \rangle \doteq 
		\langle  \theta ,  \psi \otimes \overline{\phi_j}  \rangle, \quad j \in \N^\ast.
	\end{equation}
	We claim that  $\theta_j \in (\mathcal{G}^{\sigma}(\T^n))'$. 
	Indeed, given any constant $C>0$, by Proposition \ref{prop:charGHrrho} there exists $B=B_{\sigma C}>0$ 
	such that
	\begin{equation}\label{eq:thetaseriesest}
		\begin{aligned}
			|\langle  \theta_j  ,  \psi  \rangle| & = 
			|\langle  \theta  ,  \psi \otimes \overline{\phi_j}  \rangle|  \\
			&\leq B
			\sup_{\gamma \in \Z_{+}^n}\left\{
			C^{-|\gamma|}(\gamma!)^{-\sigma} \sup_{t \in \T^n} \|\partial_t^{\gamma} \psi(t)\overline{\phi_j}\|_{H^{Mm,M\mu}}
			\right\} \\
			& = B \|\overline{\phi_j}\|_{H^{Mm,M\mu}}
			\sup_{\gamma \in \Z_{+}^n}\left\{
			C^{-|\gamma|}(\gamma!)^{-\sigma} \sup_{t \in \T^n} |\partial_t^{\gamma}\psi(t)|\right\}
			\\
			& = B \|\phi_j\|_{H^{Mm,M\mu}}
			\|\psi\|_{\mathcal{G}^{\sigma,C}} \le \widetilde{B} |\widetilde{\lambda}_j|^M \|\psi\|_{\mathcal{G}^{\sigma,C}}, \quad j \in \N^\ast,
	\end{aligned} 
	\end{equation}
	which proves the assertion, recalling Remarks \ref{rem:normphij} and \ref{rem:conjugates}. This allows us to decompose any $\theta\in\Fc'$ into a series of tensor products,
	whose first factors satisfy the estimates \eqref{eq:thetaseriesest},  
	as shown in the subsequent Lemma \ref{lem:thetaseries}. Next, we study convergence in $\Fc'$.
	%
	%
	\begin{lemma}\label{lem:thetaseries}
		Let $\theta \in \Fc'(\T ^n\times \R^d)$. Then
		\begin{equation}\label{eq:defserie}
			\theta=\sum_{j \in \N^\ast}\theta_j\otimes\phi_j,
		\end{equation}
		with $\{\theta_j\}_{j \in \N^\ast} \subset (\mathcal{G}^{\sigma}(\T^n))'$ given by \eqref{partial_in_F'}, and so satisfying \eqref{eq:thetaseriesest}.
	\end{lemma}
	\begin{proof}
		Since also $\{\overline{\phi_j}\}_{j \in \N^\ast}$ is an orthonormal basis, for any $f\in\Fc$ we can write
		\[		
			f (t,x)= \sum_{j \in \N^\ast}(f(t),\overline{\phi_j})\,\overline{\phi_j(x)}=\left[\int_{\R^d}f(t,y)\phi_j(y)\,dy\right]\!\overline{\phi_j(x)}.
		\]
		%
		Then, recalling the properties of tensor products of distributions, for any $f\in\Fc$,
		\begin{align*}
			\sum_{j \in \N^\ast} \langle  \theta_j\otimes\phi_j  , f \rangle 
			&=\sum_{j \in \N^\ast}\left\langle \theta_j  , \, \int_{\R^d} f(\cdot,y) \phi_j(y)\, dy \right\rangle 
			\\
			&=\sum_{j \in \N^\ast}\left\langle \theta, (f(\cdot),\overline{\phi_j})\otimes \overline{\phi_j}(\cdot\cdot)\right\rangle
			\\
			&=\left\langle \theta, \sum_{j \in \N^\ast}(f(\cdot),\overline{\phi_j})\otimes \overline{\phi_j}(\cdot\cdot)\right\rangle
			=\langle \theta, f\rangle,
		\end{align*}
		as claimed.
	\end{proof}
	\begin{thm}\label{thm:cauchyseq}
		Let $\{\tau_j\}_{j \in \N^\ast}\subset\Fc'(\T ^n\times \R^d)$ be such that
		$\{\langle \tau_j,f \rangle \}_{j \in \N^\ast}$ is a Cauchy sequence in $\C$, for all 
		$f \in \Fc(\T ^n\times \R^d)$. Then there exists $\tau \in \Fc'(\T ^n\times \R^d)$ such that
		$\displaystyle\tau = \lim_{j \to \infty}  \tau_j$, that is, 
		\[
			\langle \tau, f \rangle = \lim_{j \to \infty}  \langle \tau_j, f \rangle, \quad f \in \Fc(\T ^n\times \R^d).
		\]
	\end{thm}
	\begin{proof}
		By hypothesis, it is well defined the linear map $\tau:  \Fc \to \C$ given by
		\[
			\langle \tau , f \rangle = \lim_{j \to \infty} \langle \tau_j , f \rangle,\quad f\in\Fc.
		\]
		In order to verify the continuity, we first recall that
		\[
			\Fc' = \bigcup_{r,\rho \in \R} \mathcal{G}\mathcal{H}_{r,\rho}' = \bigcup_{r,\rho \in \R} 
			\left[
				\bigcap_{\sigma>1} 
				(\mathcal{G}^{\sigma}\mathcal{H}_{r,\rho}^{\sigma-1})'
			\right]\! .
		\]
		Hence, for each $j \in \N^\ast$ we can find $r_j, \rho_j \in \R$ such that 
		$\tau_j \in (\mathcal{G}^{\sigma}\mathcal{H}_{r_j,\rho_j}^{\sigma-1})'$ for any $\sigma>1$.
		Now, let $\{f_\ell\}_{\ell \in\N}$ be a sequence in $\Fc$ converging to $0$. 
		This means that for every $r,\rho \in \R$ there is $\sigma=\sigma_{r,\rho}>1$ such that 
		\[
			f_\ell \to 0 \ \textrm{ in } \  \mathcal{G}^{\sigma_{r,\rho}} \mathcal{H}^{\sigma_{r,\rho}-1}_{r,\rho}.
		\]
		In particular, $\{f_\ell\}_{\ell \in\N}$ converges to zero in each $\mathcal{G}^{\sigma_{r_j,\rho_j}} \mathcal{H}^{\sigma_{r_j,\rho_j}-1}_{r,\rho_j}$.
		Consider the complete metrizable space
		\[
			\mathcal{B} = \bigcap_{j \in \N^\ast} \mathcal{G}^{\sigma_{r_j,\rho_j}} \mathcal{H}^{\sigma_{r_j,\rho_j}-1}_{r,\rho_j} \subset \Fc,
		\]
		and define by $\omega_j$ the restrictions 
		\[
			\omega_j = \tau_j|_{\mathcal{B}} : \mathcal{B} \to \C, \ j \in \N^\ast,
		\]
		which clearly are linear and continuous. Notice that $\{\langle \omega_j , b \rangle\}$ is a Cauchy sequence in $\C$, for every $b \in \mathcal{B}$. 
		Therefore, by the  Banach-Steinhaus theorem, $\{\omega_j\}_{j \in \N^\ast}$ is equicontinuous. Since $f_{\ell}$ converges 
		to zero in $\mathcal{B}$, for every $\epsilon>0$ there exists some $\ell_0 \in \N$ such that
		\begin{equation}\label{eq:C1}
			|\langle \omega_j , f_\ell \rangle| \leq \dfrac{\epsilon}{2}, \quad \ell \ge \ell_0,  j \in \N^\ast.
		\end{equation}
		Now, by $\langle \tau , f_{\ell} \rangle = \lim_{j\to\infty} \langle \tau_j , f_{\ell} \rangle$, we may assume that for 
		$\ell\geq\ell_0$ there exists an index $j_{\ell}$ satisfying 
		\begin{equation}\label{eq:C2}
			|\langle \tau , f_\ell \rangle - \langle \tau_{j_{\ell}} , f_\ell \rangle | < \dfrac{\epsilon}{2}.
		\end{equation}
		Finally, for all $\ell\geq\ell_0$, it follows from \eqref{eq:C1} and  \eqref{eq:C2} that
		\begin{align*}
			|\langle \tau , f_\ell \rangle| & \leq 
			|\langle \tau , f_\ell \rangle - \langle \tau_{j_{\ell}} , f_\ell \rangle | +  |\langle \tau_{j_{\ell}} , f_\ell \rangle|  \\
			& = |\langle \tau , f_\ell \rangle - \langle \tau_{j_{\ell}} , f_\ell \rangle | + |\langle \omega_{j_{\ell}} , f_\ell \rangle| \\
			& < \epsilon,
		\end{align*}
		that is, $\tau\in\Fc'$. The proof is complete.
	\end{proof}
	
	The subsequent Theorem \ref{thm:seq} provides a sufficient condition on the coefficients of an expansion with
	respect to the basis $\{\phi_j\}_{j \in \N^\ast}$ to indeed produce an element of $\Fc'$. Together with Lemma \ref{lem:thetaseries}
	it completes the characterisation of $\Fc'$ in terms of eigenfunctions expansions associated with a classical,
	self-adjoint, positive SG-operator, and is a further main result of the paper.
	
	\begin{thm}\label{thm:seq}
		Let $P \in \mathrm{Op}(S_\mathrm{cl}^{m,\mu}(\R^d))$ be an elliptic, self-adjoint, positive, 
		classical SG-operator with order components $m,\mu>0$, and denote by $\{\phi_j\}_{j \in \N^\ast}$ a basis 
		of orthonormal eigenfunctions of $P$ with corresponding eigenvalues $\{\lambda_j\}_{j \in \N^\ast}$.
		Let $\{\vartheta_j\}_{j \in \N^\ast} \subset \bigcap_{\sigma>1}(\mathcal{G}^{\sigma}(\T^n))'$ 
		be a sequence such that there exist $M \in \Z$,  $B>0$, satisfying 
		\[
			|\langle  \vartheta_j ,  \psi \rangle| \leq B \|\psi\|_{\mathcal{G}^{\sigma,C}(\T^n)}
			|\widetilde{\lambda}_{j}|^{M}, \quad j \in\N^\ast, 
		\]
		for all $\sigma>1$, $C>0$, $\psi \in \mathcal{G}^{\sigma,C}(\T^n)$. Then
		\begin{equation}\label{fourier-inv}
			\vartheta = \sum_{j \in \N^\ast} \vartheta_j \otimes\phi_j\in \Fc'(\T ^n\times \R^d).
		\end{equation}
		Moreover, 
		\[
			\langle  \vartheta_j , \psi \rangle  = \langle  \vartheta ,  \psi\otimes\overline{\phi_j} \rangle, \quad \psi \in  \mathcal{G}^{\sigma}(\T^n),j \in \N^\ast.
		\]
	\end{thm}
	\begin{proof}
		Set, for each $j \in \N^\ast$,
		\[
			\varsigma_J = \sum_{j=1}^{J} \vartheta_j \otimes\phi_j \in \Fc'.
		\]
		Then, choose $M' \in \Z$ such that 
		$\sum_{j \in \N^\ast}j^{\varrho (M-M')} < + \infty$. 
		Recall that, by Corollary \ref{cor:suffhypo} and Remark \ref{rem:conjugates}, for any $f=\sum_{j \in \N^\ast}\widetilde{f_j}\overline{\phi_j} \in \Fc$ 
		there are $\sigma=\sigma_{M'}>1$ and $C=C_{M'}>0$ such that
		\[
			\|\widetilde{f_j}\|_{\mathcal{G}^{\sigma,C}} \leq C |\widetilde{\lambda}_j|^{-M'}, \quad j \in \N^\ast.
		\]
		Then, for arbitrary $\varepsilon>0$, suitable $J_0=J_0(\varepsilon)$, any $J>J_0$ and $\ell\ge1$,
		recalling also Remark \ref{rem:polyboundlambda},
		\begin{align*}
			|\langle \varsigma_{J+\ell} - \varsigma_J , f  \rangle|
			&\leq  \sum_{k=J+1}^{J+\ell} \sum_{j \in \N^\ast} |\langle  \vartheta_k  ,  \widetilde{f_j} \rangle| \, |\langle \phi_k,\overline{\phi_j}\rangle|
			=  \sum_{k=J+1}^{J+\ell} \sum_{j \in \N^\ast} |\langle  \vartheta_k  ,  \widetilde{f_j} \rangle| \, |( \phi_k,\phi_j)|
			\\&=\sum_{k=J+1}^{J+\ell} |\langle  \vartheta_k  ,  \widetilde{f_k} \rangle|
			\leq B\sum_{k=J+1}^{J+\ell} \|\widetilde{f_k}\|_{\mathcal{G}^{\sigma,C}} |\widetilde{\lambda}_k|^M \\
			&\leq BCK^{M-M'} \sum_{k=J+1}^{J+\ell}  j^{\varrho(M-M')}<\varepsilon.
		\end{align*}
		Hence, $\{\langle\varsigma_J,f\rangle\}_{j \in \N^\ast}$ is a Cauchy sequence in $\C$. Therefore, 
		by Theorem \ref{thm:cauchyseq}, there exists $\vartheta \in \Fc'$
		such that
		\[
			\langle \vartheta , f \rangle = \lim_{J\to\infty}  \langle  \varsigma_J ,f \rangle,  \; f \in  \Fc 
			\iff \vartheta=\sum_{j \in \N^\ast}\vartheta_j\otimes\phi_j.
		\]
		It also follows, for any $\psi \in \mathcal{G}^{\sigma}$, $j \in \N^\ast$,
		\begin{align*}
			\langle\vartheta,\psi\otimes\overline{\phi_j}\rangle&=\sum_{k\in\N}\langle \vartheta_k\otimes\phi_k,\psi\otimes\overline{\phi_j}\rangle
			= \sum_{k\in\N}\langle \vartheta_k, \psi\rangle \cdot \langle \phi_k,\overline{\phi_j}\rangle
			\\
			&= \sum_{k\in\N}\langle \vartheta_k, \psi\rangle \cdot ( \phi_k,\phi_j ) = \langle\vartheta_j,\psi\rangle.
		\end{align*}
		The proof is complete.
	\end{proof}

	\section{Periodic evolution equations and hypoellipticity on $\T\times\R^d$}\label{sec:torussghypoell}
	\setcounter{equation}{0}
	
In this last section we discuss the global hypoellipticity on $\T\times\R^d$ of the operator 
\begin{equation}\label{L_constant_coeff}
	L = D_t + \omega P, \ t \in \T, x \in \R^d,
\end{equation}
where $\omega = \alpha + i\beta \in \C$ and $P=\Op(p)\in\Op(S_\mathrm{cl}^{m,\mu})$ is a classical, positive, self-adjoint, elliptic SG-operator with order components $m, \mu>0$. 
More precisely, we study the global regularity on $\T\times\R^d$ of the solutions $u\in\Fc'$ of the equation $Lu =f$. Our approach is based on the expansions 
of $u\in\Fc'$ and $f\in\Fc$ generated by the operator $P$. Therefore, writing, for $t\in\T$,
\begin{equation*}
	u(t) = \sum_{j \in \N^\ast} u_j(t) \phi_j =  \sum_{j \in \N^\ast} u_j(t) \otimes \phi_j
	\ \textrm{ and } \
	f(t) = \sum_{j \in \N^\ast} f_j(t) \phi_j = \sum_{j \in \N^\ast} f_j(t) \otimes \phi_j,
\end{equation*}
we obtain that $Lu=f$ is equivalent to the (infinite) system of ODEs
\begin{equation}\label{diffe-equations}
	D_t u_j(t) +  \omega\lambda_j  u_j(t) = f_j(t), \ t \in \T, \ j \in \N^\ast.
\end{equation}
By standard arguments, the solutions of \eqref{diffe-equations} are given by 
\begin{equation}\label{fator_integrante}
	u_j(t) = u_{j0}\exp\left( -i \lambda_j\omega t \right) +  i\int_{0}^{t}\exp\left(i\lambda_j \omega (s-t) \right) f_j(s)ds,
\end{equation}
for some $u_{j0} \in \C$, $j \in \N^\ast$. 

We may assume that the coefficients  $f_j$ belong to a Gevrey class $\mathscr{G}^{\sigma}=\mathscr{G}^{\sigma}(\T)$ for all $j \in \N^\ast$. Then, 
by the properties of equation \eqref{diffe-equations}, also $u_j$ is in $\mathscr{G}^{\sigma}$ for all $j \in \N^\ast$. Now consider the set
\begin{equation}\label{eq:defZ}
	\cZ = \{j \in \N^\ast; \ \omega \lambda_{j} \in \Z\}.
\end{equation}

\begin{rem}\label{remzfin}
Notice that if $\beta \neq 0$ the set $\cZ$ is finite.
\end{rem}
If $j \notin \cZ$, then $u_{j0}$ is uniquely defined and \eqref{fator_integrante} can be written in either of the two equivalent forms
\begin{align}
	\label{Solu-1-Constant}
	u_j(t) &= \frac{i}{1 - e^{-  2 \pi i\lambda_j\omega}} \int_{0}^{2\pi}\exp\left(-i\lambda_j \omega s\right) f_j(t-s)ds
\intertext{or}
	\label{Solu-2-Constant}
	u_j(t) &= \frac{i}{e^{ 2 \pi i \lambda_j \omega} - 1} \int_{0}^{2\pi}\exp\left(i\lambda_j \omega s \right) f_j(t+s)ds. 
\end{align}
\begin{defn}
	We say that the operator $L$ defined in \eqref{L_constant_coeff} is globally hypoelliptic on $\T\times\R^d$ if
	\[
		u \in \Fc'(\T\times\R^d) \ \textrm{ and } \ Lu \in \Fc(\T\times\R^d)
		\Rightarrow
		u \in \Fc(\T\times\R^d).
	\]
\end{defn}
We point out that, in view of the Fourier characterisations obtained in Section \ref{sec:torussgexp}, 
to study the solutions of $Lu=f$ we must perform an analysis of
$|\partial^k_t u_j(t)|$ as $j\to \infty$, for all $k \in \Z_+$. In particular, we need to consider 
the size of set  $\cZ$, as well  the growth of the sequences 
\[
	\Theta_j = |1 - e^{-  2 \pi i\lambda_j \omega}|^{-1} \ \textrm{ and } \ 
	\Gamma_j =  |e^{ 2 \pi i \lambda_j \omega} - 1|^{-1},
\]
as $j \to \infty$. Loosely speaking, the behaviour of $\{\Theta_j\}_{j \in \N^\ast}$ or $\{\Gamma_j\}_{j \in \N^\ast}$
must not destroy the growth of $f_j$ and its derivatives.  
For instance, when $\beta \neq 0$  it is easy to see that both sequences converge to $1$ as  $j \to \infty$.  
On the other hand, the case $\beta=0$ is more delicate, and is connected with the so-called Diophantine approximations. 
To deal with this situation, we need the next Definition \ref{def:condA}.
\begin{defn}[Condition \condA]\label{def:condA}
We say that  a real number $\alpha$ satisfies Condition \condA\ if there are positive constants $\epsilon$ and  $C$  such that 
\[
	|\tau -\alpha \lambda_{j}| \geq C j^{-\epsilon}, 
\]
for all $(j,\tau) \in \N^\ast \times \Z$.
\end{defn}
If Condition \condA\ fails we may obtain, for any $C>0$ and any $\delta>0$,
a subsequence $\{\lambda_{j_k}\}_{k \in \N}$ and a sequence 
$\{\tau_k\}_{k\in\N} \subset \Z$, depending on $C$ and $\delta$, such that
\begin{equation}\label{ABfails}
	|\alpha \lambda_{j_k} - \tau_k | < C {j_k}^{-\delta}, \quad k\in\N^\ast.
\end{equation}
Indeed, it is enough to choose a suitable decreasing sequence $\{C_k\}_{k\in\N^\ast}\subset(0,C)$
and a suitable increasing sequence $\{\delta_k\}_{k\in\N^\ast}\subset(\delta,+\infty)$ such that $C_k\to0$ and $\delta_k\to\infty$
for $k\to\infty$, see, e.g., Appendix \ref{subs:evsubs}.
Actually, it is possibile to obtain a better statement, as we now show.
\begin{lemma}\label{lemma:useq}
If Condition \condA\ fails, there are a subsequence $\{\lambda_{j_k}\}_{k \in \N^\ast}$ and a sequence $\{\tau_k\}_{k\in\N^\ast} \subset \Z$ such that, for any $C>0$ and any $n \in \N$,
	\begin{equation}\label{univ-seq}
		|\alpha \lambda_{j_k} - \tau_k | < C j_{k}^{-n}, \quad k \to \infty.
	\end{equation}
\end{lemma}
\begin{proof}
	We know, by \eqref{ABfails}, that for any $C>0$ and $n \in \N$ there exist a subsequence $\{\lambda_{j_\ell}^{(n)}\}_{\ell \in \N^\ast}$ 
	and a sequence $\{\tau^{(n)}_{\ell}\}_{\ell \in \N^\ast} \subset \Z$ such that
	\begin{equation}
		|\alpha \lambda_{j_\ell}^{(n)} - \tau^{(n)}_{\ell}| < C j_{\ell}^{-n}, \quad\ell \in \N^\ast.
	\end{equation}
	Then, using a diagonalization argument, we define
	the subsequence $\{\lambda_{j_k}\}_{k \in \N^\ast}$ and the sequence $\{\tau_k\}_{k \in \N^\ast} \subset \Z$ as
	\[
		(  \lambda_{j_k}, \tau_k )\doteq \left(\lambda_{j_k}^{(k)}, \tau_k^{(k)}\right).
	\]
	Hence, for any $n \in \N$ we get $|\alpha \lambda_{j_k} - \tau_k | < C j_k^{-n}$, provided that $k \geq n$. This proves the claim.
\end{proof}
\noindent
The next result will be useful in the final step of the proof of Theorem \ref{thm:MainTheormHypo} below.
\begin{lemma}\label{lem:1-exp}
	Let  $\{\beta_j \}_{j \in \N^\ast}$ be a sequence of real numbers. Then, for each $j \in \N^\ast$ there exist $l(j) \in \Z$ such that
	\[
		|1 - e^{2\pi i \beta_j}| \geqslant 4 \ | \beta_j + l(j)|.
	\]
\end{lemma} 
\begin{proof}
	See Proposition 5.7 in \cite{AGK18}.
\end{proof}
	We can now characterise the global hypoellipticity on $\T\times\R^d$ of the operator $L$ in \eqref{L_constant_coeff}. The subsequent 
	Theorem \ref{thm:MainTheormHypo} is a further main result of this paper.
	\begin{thm}\label{thm:MainTheormHypo}
		The operator $L$ defined in \eqref{L_constant_coeff} is globally hypoelliptic on $\T\times\R^d$
		if and only if either $\beta \neq 0$ or $\beta=0$ and $\alpha$ satisfies Condition \condA.
	\end{thm}
	
	First, we observe that if $\beta \neq 0$, by the previous considerations, \eqref{Solu-1-Constant} or \eqref{Solu-2-Constant} imply
	\[
		\sup_{t\in\T}\left|\partial_t^\gamma u_j(t)\right| \lesssim \sup _{t \in \T}\left|\partial_t^\gamma f_j(t)\right|,\quad \gamma\in\Z_+,
	\]
	showing that $u$ satisfies the same Gevrey estimates of $f$. The global hypoellipticity of $L$ on $\T\times\R^d$ then follows,
	in view of Corollary \ref{cor:estpow}. Hence, we need focusing our analysis only on the case $\beta = 0$. 

	\smallskip

	Second, when $\beta=0$ the next Lemma \ref{lem:Lemma_Z_inf} provides 
	a necessary condition for the global hypoellipticity of $L$ on $\T\times\R^d$.
	\begin{lemma}\label{lem:Lemma_Z_inf}
		If $\beta=0$ and the set $\cZ$ defined in \eqref{eq:defZ} is infinite, the operator $L$ defined in 
		\eqref{L_constant_coeff} is not globally hypoelliptic on $\T\times\R^d$. 
	\end{lemma}
	\begin{proof}
		Since $\beta = 0$ we consider $L = D_t + \alpha P$. By the hypothesis that $\cZ$ is infinite, there is 
		a subsequence  $\{\lambda_{j_k}\}_{k \in \N^\ast}$ such that $\alpha \lambda_{j_k} \in \Z$. Then,
		let $\{u_j\}_{j \in\N^\ast}\subset\cinf(\T)$ be defined, for any $j \in \N^\ast$, by
		\[
		u_j(t) =
		\begin{cases}
			\exp \left(-i \alpha  \lambda_{j_k} t \right), & j= j_k, k\in\N^\ast, \\
			0, & j \neq j_k, k\in\N^\ast.
		\end{cases}
		\]
		Clearly, $u_{j} \in \mathcal{G}^{\sigma}(\T)$ for every $\sigma>1$, and 		
		\[
			D_t u_{j}(t) + \lambda_{j} \alpha u_{j}(t) = 0, \quad j \in \N^\ast.
		\]
		Moreover, by Theorem \ref{thm:seq}, setting
		$u(t) = \sum_{j \in \N^\ast} u_j(t)\otimes \phi_j$,  it also immediately follows $u \in  \Fc'$.
		Similarly, Corollary \ref{cor:estpow} shows that $u \notin \Fc$, observing that $|u_{j_k}(t)| \equiv 1$ for all $k \in\N^\ast$. Then, since $Lu= 0$, 
		we conclude that $L$ is not globally hypoelliptic on $\T\times\R^d$.
	\end{proof}
	It follows that to conclude the proof of Theorem \ref{thm:MainTheormHypo} it is enough to prove the next Proposition \ref{prop:nec_suf_hypo}.
	\begin{prop}\label{prop:nec_suf_hypo}
		The operator $L=D_t + \alpha P$, $\alpha\in\R$, defined in \eqref{L_constant_coeff} is globally hypoelliptic on $\T\times\R^d$
		if and only if $\alpha$ satisfies Condition \condA.
	\end{prop}
	\begin{proof} We start showing sufficiency. By Lemma \ref{lem:Lemma_Z_inf}, the set $\cZ$ in \eqref{eq:defZ} must be finite. 
		Then we may consider the solutions of equations 
		\[
			D_t u_j(t)+ \alpha  \lambda_j u_j(t)=f_j(t), 
		\]
		given by \eqref{Solu-1-Constant}, for any $j$ large enough. In view of Lemma \ref{lem:1-exp}, there are $C>0$ and $\epsilon>0$ such that
		\[
			| 1- e^{-2 \pi i \alpha \lambda_j} | \geq C j^{-\epsilon}.
		\]
		For $j$ large enough and any $\gamma\in\Z_+$ we then have
		\[
			\sup_{t \in \T} |\partial^\gamma_t u_j(t)| 
			\leq \frac{2\pi}{| 1- e^{-2 \pi i \alpha \lambda_j} |} \sup_{t \in \T}|\partial^\gamma_t f_j(t)| 
			\leq C' j^{\epsilon} \sup_{t \in \T}|\partial^\gamma_t f_j(t)|, \quad t\in\T,
		\]
		for another suitable constant $C'>0$. 
		Now, let $M\in\N$ be fixed and let $N=N_M \in \N$ satisfy $-N \varrho' + \epsilon < -M$, where 
		$\varrho'$ is given by Remark \ref{rem:polyboundlambda}. For this $N$, there are  $C_N>0$ and $\sigma_N>1$ such that
		\[
			\sup_{t \in \T}|\partial^{k}_t f_j(t)| \leq C_N^{k+1} (k!)^{\sigma_N} j^{-N}, \quad k\in\Z_+,
		\]
		which implies
		\[
			\sup_{t \in \T} |\partial^\gamma_t u_j(t)| \leq C_{N_M}^{\gamma+1} (\gamma!)^{\sigma_{N_M}} j^{-M}, \quad \gamma\in\Z_+.
		\]
		By Corollary \ref{cor:suffhypo}, it follows $u \in \Fc$ and we conclude that $L$ is globally hypoelliptic on $\T\times\R^d$.
		
		\smallskip
		To prove necessity, we proceed by contradiction, that is, we assume that 
		Condition \condA\ fails and exhibit a solution $u\in\Fc'\setminus\Fc$ of $Lu=f$ with $f\in\Fc$. 
		Indeed, by Lemma \ref{lemma:useq}, there exist $\{\lambda_{j_k}\}_{k \in \N^\ast}$ and $\{\tau_k\} \subset \Z$ 
		such that for any $n \in \N$ and for any $C>0$ it holds
		\[
			0< |\tau_k - \alpha \lambda_{j_k}| < C j_k^{-n}, \quad k \to \infty.
		\]
		Consider the sequences
		\begin{align*}
			u_j(t) & =
			\begin{cases}
				e^{-i\tau_k t}, &\textrm{ if } j = j_k, k\in\N^\ast,
				\\
				0, &\textrm{ if } j\not=j_k,  k\in\N^\ast,
			\end{cases} 
		\\
			f_j(t) & =  
			\begin{cases}
				(\alpha \lambda_{j_k} - \tau_k)e^{-i\tau_k t}, &\textrm{ if } j = j_k, k\in\N^\ast,
				\\
				0, & \textrm{ if } j\not= j_k, k\in\N^\ast.
			\end{cases}
		\end{align*}
		Since $|u_{j_k}(t)| \equiv 1$, $k\in\N^\ast$, Theorem \ref{thm:seq} and Corollary
		 \ref{cor:estpow} imply $u = \sum_{j \in \N^\ast}u_j\phi_j \in \Fc'\setminus \Fc$. 
		On the other hand, it follows from Remark \ref{rem:polyboundlambda} that
		$|\lambda_{j}| \leq C j^{\varrho}$, $j \to \infty$.
		Since $|\tau_k|\leq 1 + C_1j^{\varrho}_k$, we can estimate, for $\gamma\in\Z_+$,
		\[
			|\tau_k|^{\gamma}
			= \sum_{s=0}^{\gamma}\binom{\gamma}{s} C_1^{s}j_k^{s\varrho} 
			\leq 2^\gamma C_1^{\gamma+1} j_k^{\gamma\varrho} \leq  C_2^{\gamma+1} j_k^{\gamma\varrho}, \quad k \to \infty,
		\]
		for a new constant $C_2$ not depending on $\gamma \in \Z_+$.
		Finally, given $N \in \N$ we choose $n$ such that $\gamma \varrho - n < -N$. Then, 
		\begin{align*}
			|\partial_t^{\gamma}f_{j_k}(t)| \leq C_2^{\gamma + 1}C j_k^{\gamma\varrho - n} 
			\leq C_3^{\gamma +1} j_{k}^{-N},
		\end{align*}
		for every $\gamma \in \Z_+$, which shows $f \in \Fc$. Since 
		$Lu=f$, we conclude that $L$ is not globally hyepoelliptic on $\T\times\R^d$.
	\end{proof}

	\section{Periodic evolution equations and solvability on $\T\times\R^d$ }\label{sec:torussgsolv}
	\setcounter{equation}{0}

	We now  discuss the global solvability on $\T\times\R^d$ of the operator
	\begin{equation}\label{L_constant_coeff_solva}
		L = D_t + \omega P, \ t \in \T, x \in \R^d,
	\end{equation}
	where $\omega = \alpha + i\beta \in \C$ and $P=\Op(p)\in\Op(S_\mathrm{cl}^{m,\mu})$ is a classical, positive, self-adjoint, elliptic SG-operator with order components $m, \mu>0$.
	
	First, we point out that if  $Lu=f \in\Fc$, then, by periodicity of $u_j$, we must have
	\begin{equation}\label{admissible_condition}
		\int_{0}^{2\pi}\exp\left(i\lambda_j \omega t\right) f_j(t)dt = 0,
	\end{equation}
	whenever  $j \in \cZ$, where $f(t) = \sum_{j \in \N^\ast} f_j(t) \phi_j$.
	Hence, we introduce the space of \textit{admissible functions}
	\begin{equation}\label{def_E}
	\mathbb{E} = \{f \in \Fc; \ \eqref{admissible_condition}  \ \textrm{ holds, whenever } j \in \cZ\}.
	\end{equation}

	\begin{defn}
		We say that the operator $L$ defined in \eqref{L_constant_coeff_solva}
		is globally solvable on $\T \times \R^d$ if  for every $f \in \mathbb{E}$ there exists  $u \in \Fc'$ such that $Lu=f$.
	\end{defn}
	
	The connection between hypoellipticity and solvability is given by the next Proposition \ref{prop:GHimpGS}.
	
	\begin{prop}\label{prop:GHimpGS}
		If $L$ is globally hypoelliptic, then it is globally solvable.
	\end{prop}
	
	\begin{proof}
	By Theorem 	\ref{thm:MainTheormHypo} we have either  $\beta \neq 0$, or $\alpha$ satisfies Condition \condA. 
	Employing eventually \eqref{Solu-2-Constant} in place of \eqref{Solu-1-Constant}, we can assume, without loss of generality, that $\beta \leq 0$, . 
		
		Now, let $f \in \mathbb{E}$ be fixed. If  $j \notin \cZ$, we define
		\begin{equation}\label{sol:j_notin_Z}
			u_j(t) = \frac{i}{1 - e^{-  2 \pi i\lambda_j\omega}} \int_{0}^{2\pi}\exp\left(-i\lambda_j \omega s\right) f_j(t-s)ds,
		\end{equation}
		and for  $j \in \cZ$ we put
		\begin{equation}\label{sol:j_in_Z}
			u_j(t) = \exp(-i\lambda_j \omega t)\int_{0}^{t}\exp\left(i\lambda_j \omega s\right) f_j(s)ds.
		\end{equation}
		
		Clearly, each $u_j(t)$ is a Gevrey function on $\T$ and
		$$
		D_t u_j(t) +    \lambda_j (\alpha + i \beta) u_j(t) = f_j(t),  \ \forall j \in \N^\ast.
		$$

		Since $\cZ$ is a finite set, cf. Remark \ref{remzfin} and Lemma \ref{lem:Lemma_Z_inf},
		estimates for $u_j(t)$ in the case $j \in \cZ$ have no influence. For $j \notin \cZ$ we may use similar arguments as in the proof of
		Proposition \ref{prop:nec_suf_hypo} to conclude that
		$u(t) = \sum_{j \in \N^\ast} u_j(t) \phi_j \in \Fc.$ Then, by $Lu=f$, we see that $L$ is globally solvable.
	\end{proof}
	
	To study solvability, we need the following Definition \ref{def:condB}.
	\begin{defn}[Condition \condB]\label{def:condB}
		We say that  a real number $\alpha$ satisfies Condition \condB\ if there are positive constants $\epsilon$ and  $C$  such that
		\[
		|\tau -\alpha \lambda_{j}| \geq C j^{-\epsilon},
		\]
		for all $(j,\tau) \in \N^\ast \times \Z$ such that $\tau -\alpha \lambda_{j} \neq 0$.
	\end{defn}

	\begin{prop}\label{prop:NecCondSolv}
		If the operator $L$  is globally solvable on $\T\times\R^d$, then either $\beta \neq 0$ or $\beta=0$ and $\alpha$ satisfies Condition \condB.
	\end{prop}
	
	\begin{proof}
		We employ a contradiction argument, that is suppose that $L = D_t + \alpha P$ and 
		$\alpha$ does not satisfy Condition \condB. By Appendix \ref{subs:evsubs}
		there is a sequence $(j_\ell, \tau_\ell) \in \N^\ast \times \Z$ such that $(j_\ell)_\ell$ and $(|\tau_\ell|)_\ell$ are increasing and
		\begin{equation*}
			0<|\tau_{\ell} - \alpha\lambda_{j_\ell}|<
			j_{\ell}^{-\ell}\exp(-\ell).
		\end{equation*}
		
		Because of Remark \ref{rem:polyboundlambda} the sequence
		\begin{equation*}
			f_j(t) = \left\{
			\begin{array}{l}
				0, \ j \neq j_{\ell},  \\
				j_{\ell}^{\ell/2}\exp(\ell/2)|\tau_{\ell} - \alpha\lambda_{j_\ell}|\exp(-i\tau_{\ell}t), \ j =j_{\ell}.
			\end{array}
			\right.
		\end{equation*}
		satisfies
		\[
		|\partial_t^{\gamma}f_{j_\ell}(t)|\leq C^{\gamma+1} j_{\ell}^{-\ell/2+ \gamma \varrho}\exp(-\ell/2),
		\]
		where $C$ does not depend on $\gamma \in \Z_+$. Given $N \in \N$ we choose $\ell_0$ so that $-\ell/2+ \gamma \varrho < -N$, for all $\ell \geq \ell_0$. Then,
		$f = \sum_{j \in \N^\ast}f_j(t)\phi_j \in \Fc$. Since $f_j \equiv 0, j \in \cZ$, we conclude that $f \in \mathbb{E}$.

		Now, if  $u \in \Fc' $ is a solution of $Lu=f$ we should have
		$$
\begin{aligned}
		u_{j_\ell}(t) =& \frac{i}{1 - e^{-  2  \pi i\lambda_{j_\ell} \alpha  }} \int_{0}^{2\pi}\exp\left(-i\lambda_{j_\ell} \alpha  s \right) f_{j_\ell}(t-s)ds \\
& = \frac{i}{1 - e^{-  2  \pi i\lambda_{j_\ell} \alpha  }} j_{\ell}^{\ell/2}|\tau_{\ell} - \alpha\lambda_{j_\ell}| \int_{0}^{2\pi}\exp\left(-i\lambda_{j_\ell} \alpha  s \right)\exp(-i\tau_{\ell}(t-s))ds \\
& =  - j_{\ell}^{\ell/2}\exp(\ell/2)\dfrac{|\tau_{\ell} - \alpha\lambda_{j_\ell}|}{\tau_{\ell} - \alpha\lambda_{j_\ell}} \exp(-i\tau_{\ell}t).
\end{aligned}
		$$
		Choosing $\sigma>1$, $M\in\N$,
		and $\psi \in \mathcal G^\sigma(\mathbb T)$ defined as
		\[
			\psi(t)= (2\pi)^{-1} \sum_{\ell \in \N^*}  \exp(-\ell/4) \exp(i\tau_l t),
		\]
		we see directly that 
\[
	|\langle u_{j_\ell}, \psi \rangle| \lambda_{j_\ell}^{-M} =  \lambda_{j_\ell}^{-M} j_{\ell}^{\ell/2} \exp(\ell/4)  \to \infty, \quad \text{ as }\ell \to \infty,
\]
	which contradicts \eqref{eq:thetaseriesest}. So, $u\notin\Fc'$, which shows that $L$ is not globally solvable on $\T\times\R^d$.
	\end{proof}
	The next Theorem \ref{thm:MainTheormSolv} is the last main result of this paper.
	\begin{thm}\label{thm:MainTheormSolv}
		The operator $L$ defined in \eqref{L_constant_coeff_solva} is globally solvable on $\T\times\R^d$
		if and only if either $\beta \neq 0$ or $\beta=0$ and $\alpha$ satisfies Condition \condB.
	\end{thm}
In view of Proposition \ref{prop:NecCondSolv}, we only need to prove 
that $\beta=0$ and $\alpha$ satisfying Condition \condB\ imply global solvability, which we do in the subsequent Proposition \ref{prop:finalstep}.  
	\begin{prop}\label{prop:finalstep}
		If $\beta = 0$  and $\alpha$ satisfies Condition \condB, then $L$
		is globally solvable.
	\end{prop}
	\begin{proof}
		Consider $f \in \mathbb{E}$ be fixed. We set
		\begin{equation}\label{sol:MainThm_j_in_Z}
			u_j(t) = \exp(-i\lambda_j \alpha t)\int_{0}^{t}\exp\left(i\lambda_j \alpha s\right) f_j(s)ds,
		\end{equation}
		in case $j \in \cZ$ and
		\begin{align}\label{Sol-beta=0-jnotinZ}
			u_j(t) = \frac{i}{1 - e^{-  2  \pi i\lambda_j \alpha  }} \int_{0}^{2\pi}\exp\left(-i\lambda_j \alpha s \right) f_j(t-s)ds,
		\end{align}
		if $j \notin \cZ$. Then, 
		$$
		|u_j(t)| \leq 2\pi \sup_{t \in \mathbb{T}}|f_j(t)|, \ j \in \cZ,
		$$
		and
		$$
		|u_j(t)| \leq C j^{\epsilon} \sup_{t \in \mathbb{T}}|f_j(t)|, \ j \notin \cZ.
		$$

		It follows from Theorem \ref{thm:seq} that $u(t) = \sum_{j \in \N^\ast}u_j(t) \phi_j \in \Fc'$ and $Lu=f$.

	\end{proof}

	\appendix

	\section{The calculus of SG pseudo-differential operators}\label{sec:sgcalc}
	\setcounter{equation}{0}

We here recall some basic definitions and facts about the $SG$-calculus of pseudo-differential 
operators, through
standard material appeared, e.g., in \cite{ACS19b,CPS1} and elsewhere (sometimes with slightly different notational choices).
A detailed description of the calculus can be found in \cite{cordes}.

The class $S ^{m,\mu}=S ^{m,\mu}(\R^{d})$ of $SG$ symbols of order $(m,\mu) \in \R^2$ is given by all the functions 
$a(x,\xi) \in \mathscr C^\infty(\R^d\times\R^d)$
with the property
that, for any multiindices $\alpha,\beta \in \N_0^d$, there exist
constants $C_{\alpha\beta}>0$ such that the conditions \eqref{eq:disSG}
hold (see \cite{cordes,ME,PA72}). We often omit the base spaces $\R^d$, $\R^{2d}$, etc., from the notation.

For $m,\mu\in\R$, $\ell\in\N_0$,
\[
	\vvvert a \vvvert^{m,\mu}_\ell
	= 
	\max_{|\alpha+\beta|\le \ell}\sup_{x,\xi\in\R^d}\x^{-m+|\alpha|} 
	                                                                     \csi^{-\mu+|\beta|}
	                                                                    | \partial^\alpha_x\partial^\beta_\xi a(x,\xi)|, \quad a\in\ S^{m,\mu},
\]
is a family of seminorms, defining  the Fr\'echet topology of $S^{m,\mu}$.

The corresponding
classes of pseudo-differential operators $\Op (S ^{m,\mu})=\Op (S ^{m,\mu}(\R^d))$ are given by
\begin{equation}\label{eq:psidos}
	(\Op(a)u)(x)=(a(\cdot,D)u)(x)=(2\pi)^{-d}\int e^{\ii x\xi}a(x,\xi)\hat{u}(\xi)d\xi,  
\end{equation}
where $a\in S^{m,\mu}(\R^d),u\in\mathscr S(\R^d)$ and extended by duality to $\mathscr S^\prime(\R^d)$.
The operators in \eqref{eq:psidos} form a
graded algebra with respect to composition, that is,
$$
\Op (S ^{m_1,\mu _1})\circ \Op (S ^{m_2,\mu _2})
\subseteq \Op (S ^{m_1+m_2,\mu _1+\mu _2}).
$$
The symbol $c\in S ^{m_1+m_2,\mu _1+\mu _2}$ of the composed operator $\Op(a)\circ\Op(b)$,
$a\in S ^{m_1,\mu _1}$, $b\in S ^{m_2,\mu _2}$, admits the asymptotic expansion
\begin{equation}
	\label{eq:comp}
	c(x,\xi)\sim \sum_{\alpha}\frac{i^{|\alpha|}}{\alpha!}\,D^\alpha_\xi a(x,\xi)\, D^\alpha_x b(x,\xi),
\end{equation}
which implies that the symbol $c$ equals $a\cdot b$ modulo $S ^{m_1+m_2-1,\mu _1+\mu _2-1}$.

Notice that
\[
	 S ^{-\infty,-\infty}=S ^{-\infty,-\infty}(\R^{d})= \bigcap_{(m,\mu) \in \R^2} S ^{m,\mu} (\R^{d})
	 =\mathscr S(\R^{2d}),
\]
and, by definition,
\[
	S^{-\infty,-\infty}\subset S^{m,\mu}\subset S^{s,\sigma}, \quad m,s,\mu,\sigma\in\R, m\le s, \mu\le\sigma.
\]
For any $a\in S^{m,\mu}$, $(m,\mu)\in\R^2$,
$\Op(a)$ is a linear continuous operator from $\mathscr S(\R^d)$ to itself, extending to a linear
continuous operator from $\mathscr S^\prime(\R^d)$ to itself, and from
$H^{r,\rho}(\R^d)$ to $H^{r-m,\rho-\mu}(\R^d)$, with $H^{r,\rho}$ defined in \eqref{eq:skspace}.
When $r\ge r^\prime$ and $\rho\ge\rho^\prime$, the continuous embedding 
$H^{r,\rho}\hookrightarrow H^{r^\prime,\rho^\prime}$ holds true. It is compact when $r>r^\prime$ and $\rho>\rho^\prime$.
Since $H^{r,\rho}=\jap^{-r}\,H^{0,\rho}=\jap^{-r}\, H^\rho$, with $H^\rho$ the usual Sobolev space of order $\rho\in\R$, we 
find $\rho>k+\dfrac{d}{2} \Rightarrow H^{r,\rho}\hookrightarrow C^k(\R^d)$, $k\in\N_0$. One actually finds
\begin{equation}
\begin{aligned}\label{eq:spdecomp}
	& \bigcap_{r,\rho\in\R}H^{r,\rho}(\R^d)=H^{\infty,\infty}(\R^d)=\mathscr S(\R^d), \\
	& \bigcup_{r,\rho\in\R}H^{r,\rho}(\R^d)=H^{-\infty,-\infty}(\R^d)=\mathscr S^\prime(\R^d),
 \end{aligned}
\end{equation}
as well as, for the space of \textit{rapidly decreasing distributions}, see \cite[Chap. VII, \S 5]{schwartz}, 
\begin{equation}\label{eq:rdd}
	\mathscr S^\prime(\R^d)_\infty=\bigcap_{z\in\R}\bigcup_{\zeta\in\R}H^{z,\zeta}(\R^d).
\end{equation}
The continuity property of
the elements of $\Op(S^{m,\mu})$ on the scale of spaces $H^{r,\rho}(\R^d)$, $(m,\mu),(r,\rho)\in\R^2$, is expressed 
more precisely in the next theorem.
\begin{thm}[{\cite[Chap. 3, Theorem 1.1]{cordes}}] \label{thm:sobcont}
	Let $a\in S^{m,\mu}(\R^d)$, $(m,\mu)\in\R^2$. Then, for any $(r,\rho)\in\R^2$, 
	$\Op(a)\in\mathcal{L}(H^{r,\rho}(\R^d),H^{r-m,\rho-\mu}(\R^d))$, and there exists a constant $C>0$,
	depending only on $d,m,\mu,r,\rho$, such that
	\begin{equation}\label{eq:normsob}
		\|\Op(a)\|_{\scrL(H^{r,\rho}(\R^d), H^{r-m,\rho-\mu}(\R^d))}\leq 
		C\vvvert a \vvvert_{\left[\frac{d}{2}\right]+1}^{m,\mu},
	\end{equation}
	where $[t]$ denotes the integer part of $t\in\R$.
\end{thm}
The class $\caO(m,\mu)$ of the \textit{operators of order $(m,\mu)$} is introduced as follows, see, e.g., \cite[Chap. 3, \S 3]{cordes}.
\begin{defn}\label{def:ordmmuopr}
	A linear continuous operator $A\colon\caS(\R^d)\to\caS(\R^d)$
	belongs to the class $\caO(m,\mu)$, $(m,\mu)\in\R^2$, of the operators of order $(m,\mu)$ if, for any $(r,\rho)\in\R^2$,
	it extends to a linear continuous operator $A_{r,\rho}\colon H^{r,\rho}(\R^d)\to H^{r-m,\rho-\mu}(\R^d)$. We also define
	\[
		\caO(\infty,\infty)=\bigcup_{(m,\mu)\in\R^2} \caO(m,\mu), \quad
		\caO(-\infty,-\infty)=\bigcap_{(m,\mu)\in\R^2} \caO(m,\mu).		
	\]
\end{defn}
\begin{rem}\label{rem:O}
	\begin{enumerate}
		\item[(i)] Trivially, any $A\in\caO(m,\mu)$ admits a linear continuous extension 
		$A_{\infty,\infty}\colon\caS^\prime(\R^d)\to\caS^\prime(\R^d)$. In fact, in view of \eqref{eq:spdecomp}, it is enough to set
		$A_{\infty,\infty}|_{H^{r,\rho}(\R^d)}= A_{r,\rho}$.
		\item[(ii)] Theorem \ref{thm:sobcont} implies $\Op(S^{m,\mu}(\R^d))\subset\caO(m,\mu)$, $(m,\mu)\in\R^2$.
		\item[(iii)] $\caO(\infty,\infty)$ and $\caO(0,0)$ are algebras under operator multiplication, $\caO(-\infty,-\infty)$ is an ideal
		of both  $\caO(\infty,\infty)$ and $\caO(0,0)$, and 
		$\caO(m_1,\mu_1)\circ\caO(m_2,\mu_2)\subset\caO(m_1+m_2,\mu_1+\mu_2)$.
	\end{enumerate}
\end{rem}
\noindent
The following characterisation of the class $\caO(-\infty,-\infty)$ is often useful.
\begin{prop}[{\cite[Ch. 3, Prop. 3.4]{cordes}}] \label{thm:smoothing}
	The class $\caO(-\infty,-\infty)$ coincides with $\Op(S^{-\infty,-\infty}(\R^d))$ and with the class of smoothing operators,
	that is, the set of all the linear continuous operators $A\colon\caS^\prime(\R^d)\to\caS(\R^d)$. All of them coincide with the
	class of linear continuous operators $A$ admitting a Schwartz kernel $k_A$ belonging to $\caS(\R^{2d})$. 
\end{prop}
An operator $A=\Op(a)$ and its symbol $a\in S ^{m,\mu}$ are called \emph{elliptic}
(or $S ^{m,\mu}$-\emph{elliptic}) if \eqref{eq:sgell} holds true.
If $R=0$, $a^{-1}$ is everywhere well-defined and smooth, and $a^{-1}\in S ^{-m,-\mu}$.
If $R>0$, then $a^{-1}$ can be extended to the whole of $\R^{2d}$ so that the extension $\widetilde{a}_{-1}$ satisfies $\widetilde{a}_{-1}\in S ^{-m,-\mu}$.
An elliptic $SG$ operator $A \in \Op (S ^{m,\mu})$ admits a
parametrix $A_{-1}\in \Op (S ^{-m,-\mu})$ such that
\[
A_{-1}A=I + R_1, \quad AA_{-1}= I+ R_2,
\]
for suitable $R_1, R_2\in\Op(S^{-\infty,-\infty})$, where $I$ denotes the identity operator. 
In such a case, $A$ turns out to be a Fredholm
operator on the scale of functional spaces $H^{r,\rho}$,
$(r,\rho)\in\R^2$. If $A\in\Op(S^{m,\mu})$ is elliptic and $B\in\Op(S^{m',\mu'})$ with $m'<m,\mu'<\mu$, then $A+B$ is also elliptic.

 
For $r,\rho \in \mathbb R$, set $\lambda_{r,\rho}(x,\xi)=\lambda_r(x)\cdot\lambda_\rho(\xi)=\langle x \rangle^{r} \langle \xi \rangle^{\rho}$ 
and consider the corresponding operator $\Lambda_{r,\rho}=\mathrm{Op}(\lambda_{r,\rho})$, the standard \textit{SG order reductions}. 	
In addition to $\Lambda_{r,\rho}$, one can also consider the SG order reductions defined by
\[
	\Pi_{m,\mu} = \langle D\rangle^{\mu/2} \langle \cdot\rangle^{m} \langle D \rangle^{\mu/2}, \quad m,\mu\in\R.
\]
Then, $\Pi_{m,\mu}$ is self-adjoint and elliptic (actually, invertible). Indeed, the operators $\langle D \rangle^\mu$ and $\langle \cdot \rangle^m$ 
are self-adjoint, which yields that $\Pi_{m,\mu}$ is self-adjoint as well. Ellipticity is straightforward. Moreover, for $m,\mu>0$, $u\in H^{m,\mu}$,
\[
(\Pi_{m,\mu} u, u) = ( \langle \cdot \rangle^{m/2} \langle D \rangle^{\mu/2} u,  \langle \cdot \rangle^{m/2} \langle D \rangle^{\mu/2} u) 
= \| u \|^2_{H^{m/2,\mu/2}} \geq 0,
\]
which shows that $\Pi_{m,\mu}$ is positive. 

Fix now $m, \mu \in\N\setminus\{0\}$ and consider a partial differential operator $P$ of the form
\begin{equation}\label{eq:opdiff}
	P=p(x, D)=\sum_{\substack{|\alpha| \leq m \\|\beta| \leq \mu}} c_{\alpha\beta} x^\alpha D_x^\beta, \quad D_x^\alpha=(-i)^{|\alpha|} \partial_x^\alpha,
\end{equation}
with coefficients $c_{\alpha\beta} \in \mathbb{C}$. Moreover, suppose that $P$ 
satisfies the ellipticity property
\[
	|p(x,\xi)| \geq C \langle x \rangle^{m} \langle \xi \rangle^{\mu}, \quad (x,\xi)\in\R^d\times\R^d, |x|+|\xi| \geq R \ge0,
\]
with $C$ a real positive constant. The latter is equivalent to the fact that the principal homogeneous symbols of $P$ do not vanish, namely
(cf., e.g., \cite[Ch. 3]{NR}, (3.0.6), (3.0.7), (3.0.8)):
\[
\begin{aligned}
	[\sigma_\psi(P)](x,\xi)=[\sigma_\psi^\mu(P)](x,\xi)&=\sum_{\substack{|\alpha|\le m \\
			|\beta| = \mu}} c_{\alpha \beta} x^\alpha \xi^\beta \neq 0, \; x,\xi \in \mathbb{R}^d, \xi \neq 0, \\
	[\sigma_e(P)](x,\xi)=[\sigma_e^{m}(P)](x,\xi)&=\sum_{\substack{|\alpha| = m \\
			|\beta|\le\mu}} c_{\alpha \beta} x^\alpha \xi^\beta \neq 0, \; x,\xi \in \mathbb{R}^d, x \neq 0,\\
	[\sigma_{\psi e}(P)](x,\xi) =[\sigma_{\psi e}^{\mu,m}(P)](x,\xi) &=\sum_{\substack{|\alpha|=m \\
			|\beta|=\mu}} c_{\alpha \beta} x^\alpha \xi^\beta \neq 0, \; x,\xi \in \mathbb{R}^d, x,\xi \neq 0.
\end{aligned}
\]
Observe that a positive, elliptic (actually, invertible), self-adjoint operator $P$ of the form \eqref{eq:opdiff}
is obtained by $\Pi_{m,\mu}$ choosing $m, \mu \in 2 \N\setminus\{0\}$. We also recall that the concept of principal homogeneous
symbol in three components $\sigma(P)=(\sigma_\psi(P), \sigma_{e}(P),\sigma_{\psi e}(P))$ extends
to the subclass $S_\mathrm{cl}^{m,\mu}$ of SG-classical operators, which includes the differential operators of the
form \eqref{eq:opdiff}, see again \cite[Ch. 3]{NR} and, e.g., \cite{BC11,CD21,ManPan}.

We conclude this section recalling the proof of Lemma \ref{lemma:h-s}.
\begin{proof}[Proof of Lemma \ref{lemma:h-s}]
	Denote by $K_{M}(x,y)$ the Schwartz kernel of the operator $P^M$ and by $p_{M}(x,\xi)$ its symbol. Then,
	\begin{equation}\label{eq:kerPM}
		\begin{aligned}
		K_{M}(x,y) = b(x,x-y) &= \mathfrak F^{-1}_{\cdot \to x-y}\left( p_M (x, \cdot)\right)
		\\
		&=(2 \pi)^{-d} \int e^{i(x-y) \, \xi} p_{M}(x,\xi) \, d \xi.
		\end{aligned}
	\end{equation}
	We need to show
	$$
	\iint \left|K_{M}(x,y) \right|^2 \, dx  dy < \infty.
	$$
	By the calculus, $p_M\in S^{Mm, M \mu}\Rightarrow p_M(x, \xi) \in L^2$, since 
	$$
	\iint |p_M(x, \xi)|^2 \, dx  d \xi \lesssim \iint \langle x \rangle^{2 Mm} \langle \xi \rangle^{2 M \mu} \, dx d \xi < \infty.
	$$
	By \eqref{eq:kerPM},
	$$
	\iint \left|K_{M}(x,y) \right|^2 \, dx dy =\iint \left|b(x,x-y) \right|^2 \, dy dx = \iint \left|b(x,z) \right|^2 \, dz dx  < \infty,
	$$
	as claimed, since $b(x,\cdot)\in L^2$ for any $x\in\R^d$, being the Fourier transform of $p_M(x,\cdot)\in L^2$.
\end{proof}

	\section{Eigenvalues of elliptic SG-operators}\label{sec:sgoprsev}
	\setcounter{equation}{0}
	
	\subsection{Eigenvalues asymptotics}\label{subs:sgasym}
	Consider a classical, positive, self-adjoint, elliptic SG-operator $P\in\Op(S_\mathrm{cl}^{m,\mu})$
	with order components $m,\mu>0$. In \cite{ManPan} it has been proved that, for $\lambda \to +\infty$,
	\begin{equation}\label{eq:sgspctasm}
	N(\lambda) \sim 
	\begin{cases}
		C_1 \ \lambda^{d/\min\{m,\mu\}}, \quad & m\neq \mu, \\
		C_2 \ \lambda^{d/m} \log\lambda,& m=\mu,
	\end{cases}
	\end{equation}
	where 
	$N(\lambda)= N_P(\lambda)=\left| \{ \lambda_j \leq \lambda : \lambda_j \text{ eigenvalue of } P \} \right|$ is the spectral counting function of the operator $P$,
	and the constants $C_1,C_2$ depend on the principal symbol of $P$.  Such results are analog to other ones,
	obtained for different classes of operators and/or systems globally defined on $\R^d$, see, for instance, 
	\cite{He84,HeRo81,MaPa24,Pa10,Pa14,PaWa01} and the references quoted therein.
	Further improvements to the associated Weyl formula for SG-operators, on manifolds with cylindrical ends and, more generally,
	on asymptotically Euclidean manifolds, have been subsequently proved in \cite{BC11, CorMan13, CD21}.
	
	By the asymptotic behaviour \eqref{eq:sgspctasm} of the counting function $N$ it is possibile to obtain, as usual, the asymptotic
	behaviour of the eigenvalues $\lambda_j$. Namely, for $j\to\infty$,
	\begin{equation}\label{eq:sgeigenasm}
		\lambda_j\sim
		\begin{cases}
		\widetilde{C}_1\,  j^{\min\{m,\mu\}/d}, \quad & m\neq \mu, \\
		\widetilde{C}_2 \left(\dfrac{j}{\log j}\right)^{m/d},& m=\mu.
		\end{cases}
	\end{equation}
	
	For the sake of completeness, we include here the proof of \eqref{eq:sgeigenasm} for the case $m=\mu$.
 	This asymptotic formula and the one valid for the case $m\not=\mu$ follows by the analogous arguments valid for other operator 
	classes (see, e.g., \cite[Proposition 13.1]{Shubin}). Let us first observe that the function $f(\lambda)=(\lambda^\frac{d}{m}\log\lambda)^{-1}$
	is strictly decreasing on the interval $I=(1,+\infty)$, since $f^\prime(\lambda)<0$ on $I$, as it is immediate to check. Then, its inverse function
	$f^{-1}$ is strictly decreasing and continuous as well. By definition of $N(\lambda)$, it follows that, for sufficiently large $j$,
	\[
		(1-\varepsilon)C_2j^{-1}\le f(\lambda_j)\le (1+\varepsilon)C_2j^{-1},
	\]
	which, by continuity, implies
	\begin{align*}
		f^{-1}(C_2j^{-1})-\widetilde{\varepsilon} &\le f^{-1}((1+\varepsilon)C_2j^{-1})
		\le\lambda_j\le
		\\
		&\le f^{-1}((1-\varepsilon)C_2j^{-1})\le f^{-1}(C_2j^{-1})+\widetilde{\varepsilon},
	\end{align*}
	that is, $\lambda_j\sim f^{-1}(C_2 j^{-1})$, $j\to\infty$. To conclude, it is enough to prove that, for $t\to0^+$,
	\begin{equation}\label{eq:finvasmpt}
		f^{-1}(t)\sim h(t)=\left(-\frac{m}{d}t\log t\right)^{-\frac{m}{d}}.
	\end{equation}
	In fact, \eqref{eq:finvasmpt} implies, for $j\to\infty$,
	\[
		\lambda_j\sim f^{-1}(C_2j^{-1})\sim\left(\frac{m}{d}C_2\right)^{-\frac{m}{d}} j^\frac{m}{d}
		(\log j-\log C_2)^{-\frac{m}{d}}\sim
		\widetilde{C}_2\left(\frac{j}{\log j}\right)^\frac{m}{d},
	\]
	as claimed. To prove \eqref{eq:finvasmpt}, we compute
	\begin{align*}
		\lim_{t\to0^+}\frac{f^{-1}(t)}{h(t)}&=\lim_{\lambda\to+\infty}\frac{f^{-1}(f(\lambda))}{h(f(\lambda))}
		\\
		&=\lim_{\lambda\to+\infty}
		\frac{\lambda}
		       {\left[
		       -\frac{m}{d}\left(\lambda^\frac{d}{m}\log\lambda\right)^{-1}
		       \log\left(\lambda^\frac{d}{m}\log\lambda\right)^{-1}
		       \right]^{-\frac{m}{d}}}
		\\
		&=\lim_{\lambda\to+\infty}
		\frac{\lambda}
		       {\lambda\left[
		       \frac{m}{d}(\log\lambda)^{-1}\left(\frac{d}{m}\log\lambda+\log\log\lambda\right)
		       \right]^{-\frac{m}{d}}}
		\\
		&=\lim_{\lambda\to+\infty}
		\left(1+\frac{m}{d}\frac{\log\log\lambda}{\log\lambda}\right)^\frac{m}{d}=1.
	\end{align*}
	As usual, it is also possible to prove that \eqref{eq:sgeigenasm} implies \eqref{eq:sgspctasm}. The details are left for the reader.
	
	\subsection{Subsequences of the eigenvalues sequence $\{\lambda_j\}_{j\in\N^\ast}$}\label{subs:evsubs}
	Here we provide, for the sake of completeness, a proof of the fact that, when Condition \condA\ fails, it is
	possibile to find a subsequence $\{\lambda_{j_k}\}_{k\in\N^\ast}\subseteq\{\lambda_j\}_{j\in\N^\ast}$ and 
	a sequence $\{\tau_k\}_{k\in\N^\ast}\subset\Z$ such that, for any $k\in\N^\ast$,
	\begin{equation}\label{eq:lambdasubseq}
		|\tau_k-\alpha\lambda_{j_k}|<C_kj_k^{-k},
	\end{equation}
	with a strictly decreasing sequence $\{C_k\}_{k\in\N^ast}\subset(0,1)$.
	Without loss of generality, in this section we assume that $\lambda_j>0$ for all $j\in\N^\ast$.
	
	In the (trivial) case $\alpha=0$, choose any sequence $\{C_k\}_{k\in\N^\ast}\subset(0,1)$ such that, for any $k\in\N^\ast$,
	$C_{k+1}<C_k$, and set $\delta_k=j_k=k$, $\tau_k=0$. Clearly, with such choices \eqref{eq:lambdasubseq} holds true.
	We consider then the case $\alpha>0$ (the case $\alpha<0$ is completely similar, and its details are left to the reader).
	We assume, for simplicity, that $\alpha\lambda_j\notin\Z$ for all $j\in\N^\ast$, that is, the set $\cZ$ in \eqref{eq:defZ} is empty.
	
	We first observe that, taking $\{C_k\}_{k\in\N^\ast}\subset(0,1)$, since $\alpha\lambda_j\ge0$ for any $j\in\N^\ast$, it follows
	that $\tau_k>-1\Rightarrow\tau_k\ge0$ for all $k\in\N^\ast$. Indeed,
	\[
		|\tau_k-\alpha\lambda_{j_k}|<C_kj_k^{-k}\Leftrightarrow
		\alpha\lambda_{j_k}-C_k j_k^{-k}<\tau_k<\alpha\lambda_{j_k}+C_k j_k^{-k},
	\]
	and
	\[
		\alpha\lambda_{j_k}-C_k j_k^{-k}\ge -C_k j_k^{-k} \ge -C_k > -1.
	\]
	Set $\delta_1=1$ and $C_1=\dfrac{1}{e}$. Then, since Condition \condA\ fails, there exist $\tau_1\in\Z_+$ and $j_1\in\N^\ast$ such that
	\[
		|\tau_1-\alpha\lambda_{j_1}|<C_1j_1^{-1} = e^{-1}j_1^{-1}.
	\] 
	Now, set $\delta_2=2$ and
	\begin{equation}\label{eq:defC2}
		C_2=\dfrac{1}{e}\min\{C_1,\{|m-\alpha\lambda_k|k^2\colon k=1,\dots,j_1, m=0,\dots,\tau_1+[\alpha\lambda_{j_1}]+1\}\},
	\end{equation}
	where we denoted by $[y]$ the integer part of $y\in\R$.
	Recalling that $\cZ$ is empty, that is, $m-\alpha\lambda_k\not=0$ for any $m\in\Z_+$ and $k\in\N^\ast$,
	it follows $0<C_2<C_1<1$, and, since Condition \condA\ fails, there exist $\tau_2\in\Z_+$ and $j_2\in\N^\ast$ such that
	\[
		|\tau_2-\alpha\lambda_{j_2}|<C_2j_2^{-2} < e^{-2}j_1^{-2}.
	\] 
	If it were $j_2\le j_1$, since the sequence $\{\lambda_j\}$ is non decreasing, we would have
	\begin{align*}
		|\tau_1-\tau_2|&\le|\tau_1-\alpha\lambda_{j_1}|+\alpha(\lambda_{j_1}-\lambda_{j_2})+|\alpha\lambda_{j_2}-\tau_2|
		<C_1j_1^{-1}+\alpha\lambda_{j_1}+C_2j_2^{-2}
		\\
		&< C_1+[\alpha\lambda_{j_1}]+1+\frac{C_1}{e}<[\alpha\lambda_{j_1}]+2,
	\end{align*}
	that is, $\tau_2\in\Z\cap[0,\tau_1+[\alpha\lambda_{j_1}]+1]=:A_2$. Then, by \eqref{eq:defC2}, having
	assumed $j_2\le j_1$, for any $m\in A_2$ we have $C_2<|m-\alpha\lambda_{j_2}|j_2^2$, and $\tau_2\in A_2$ implies
	\[
		C_2 < |\tau_2-\alpha\lambda_{j_2}|j_2^2<C_2,
	\]
	which is a contradiction. Then, it must necessarily be $j_2>j_1$. 
	
	\smallskip
	
	\noindent
	Now, set $\delta_3=3$ and
	\begin{equation}\label{eq:defC3}
		C_3=\frac{1}{e}\min\{C_2,\{|m-\alpha\lambda_k|k^3\colon k=1,\dots,j_2, m=0,\dots,\tau_2+[\alpha\lambda_{j_2}]+1\}\}.
	\end{equation}
	As above,	it follows $0<C_3<C_2<C_1<1$, and, since Condition \condA\ fails, there exist $\tau_3\in\Z_+$ and $j_3\in\N^\ast$ such that
	\[
		|\tau_3-\alpha\lambda_{j_3}|<C_3j_3^{-3} < e^{-3}j_3^{-3}.
	\] 
	If it were $j_3\le j_2$, we would have
	\[
		|\tau_2-\tau_3|<[\alpha\lambda_{j_2}]+2,
	\]
	that is, $\tau_3\in\Z\cap[0,\tau_2+[\alpha\lambda_{j_2}]+1]=:A_3$. Then, by \eqref{eq:defC3}, having
	assumed $j_3\le j_2$, for any $m\in A_3$ we have $C_3<|m-\alpha\lambda_{j_3}|j_3^3$, and $\tau_3\in A_3$ implies
	\[
		C_3 < |\tau_3-\alpha\lambda_{j_3}|j_3^3<C_3,
	\]
	which is a contradiction. Then, it must necessarily be $j_3>j_2$. Iterating the above argument, by induction
	we obtain the subsequence
	$\{\lambda_{j_k}\}_{k\in\N^\ast}$ and the sequences $\{C_k\}_{k\in\N^\ast}\subset(0,1)$ and
	$\{\tau_k\}_{k\in\N^\ast}\subset\Z_+\subset\Z$ with the requested properties.  
	
	Notice that, by construction, $\tau_k-\alpha\lambda_{j_k}\to0$ for $k\to\infty$. Then, for a suitable $k_0\in\N^\star$,
	\[
		\tau_k>\alpha(\lambda_{j_k}-1), \quad k>k_0,
	\]
	giving $\tau_k\to\infty$, $k\to\infty$. By a standard argument, passing to subsequences,
	$\{\tau_k\}_{k\in\N^\ast}$ can be made strictly increasing.

\bibliographystyle{plain} %

\end{document}